\documentclass[reqno, 11pt]{amsart}
\usepackage[mathscr]{eucal}
\usepackage[all]{xy}
\usepackage{mathrsfs}
\usepackage{xypic}
\usepackage{amsfonts}
\usepackage{amsmath}
\usepackage{amsthm,bm}
\usepackage{amssymb,enumerate}
\usepackage{latexsym}
\usepackage{tabularx}
\usepackage{graphicx}
\usepackage{pict2e}
\usepackage{tikz}
\usepackage[pagebackref]{hyperref}
\usepackage[text={160mm,240mm},centering]{geometry}
\input diagxy
\input xy

\newtheorem{thm}{Theorem}[section]
\newtheorem{mthm}[thm]{Main Theorem}
\newtheorem{cor}[thm]{Corollary}
\newtheorem{conjecture}[thm]{Conjecture}
\newtheorem{lem}[thm]{Lemma}
\newtheorem{prop}[thm]{Proposition}
\newtheorem{quest}[thm]{Question}

\theoremstyle{definition}
\newtheorem{defn}[thm]{Definition}
\newtheorem{rem}[thm]{Remark}

\newtheorem{ex}[thm]{Example}
\numberwithin{equation}{section}

\newcommand{\im}{\textmd{im}}
\begin{document}

\title{Dolbeault cohomologies of blowing up complex manifolds}

\author{Sheng Rao}
\address{School of Mathematics and Statistics, Wuhan University, Wuhan 430072, P. R. China}
\email{raoshengmath@gmail.com, likeanyone@whu.edu.cn}%
\thanks{Rao is partially supported by NSFC (Grant No. 11671305, 11771339). S. Yang and X. Yang are supported by NSFC (Grant
No. 11571242, 11701414, 11701051) and the China Scholarship Council.}

\author{Song Yang}
\address{Center for Applied Mathematics, Tianjin University, Tianjin 300072, P. R. China}%
\email{syangmath@tju.edu.cn}%

\author{Xiangdong Yang}
\address{Department of Mathematics, Chongqing University, Chongqing 401331, P. R. China}
\email{xiangdongyang2009@gmail.com, math.yang@cqu.edu.cn}


\date{\today}


\begin{abstract}
We prove a blow-up formula for Dolbeault cohomologies of compact complex manifolds by introducing relative Dolbeault cohomology.
As corollaries, we present a uniform proof for bimeromorphic invariance of $(\bullet,0)$- and $(0,\bullet)$-Hodge numbers on a compact complex manifold,
and obtain the equality for the numbers of the blow-ups and blow-downs in the weak factorization of the bimeromorphic map between two compact complex manifolds with equal $(1,1)$-Hodge number or equivalently second Betti number.
Many examples of the latter one are listed.
Inspired by these, we obtain the bimeromorphic stability for degeneracy of the Fr\"olicher spectral sequences at $E_1$ on compact complex threefolds and fourfolds.
\end{abstract}

\maketitle
\setcounter{tocdepth}{1}

\tableofcontents

\section{Introduction}
In algebraic geometry and differential geometry, Dolbeault cohomology (named after Pierre Dolbeault) is an analog of de Rham cohomology for complex manifolds.
For a complex manifold $X$, its Dolbeault cohomology groups $H_{\bar{\partial}}^{p,q}(X)$ depend on a pair of integers $p$ and $q$,
and are realized as a subquotient of the space of complex differential forms of degree $(p,q)$.
In complex geometry, blow-up or blow-down is a type of geometric transformation which replaces a subspace of a given complex space with all the directions pointing out of that subspace.
In particular, blow-up is the most fundamental transformation in birational geometry,
because every birational projective morphism is a blow-up morphism with a possibly singular center.
Besides its importance in describing birational transformations,
the blow-up also provides us with an important way of constructing new complex spaces.
For instance, most procedures for resolution of singularities proceed by blowing up singularities until they become smooth.

Due to the de Rham Theorem it is known that the de Rham cohomology of a smooth manifold is a topological invariant.
Compared with the de Rham cohomology, the Dolbeault cohomology of a complex manifold depends on the complex structures,
i.e., it is a biholomorphic invariant.
On the geometry of the blow-up of a complex manifold with a smooth center,
the de Rham blow-up formula shows the variant of de Rham cohomology under the blow-up transformations.
In literatures, there are many different versions of blow-up formulas for various (co)homology theories;
for instance, the cyclic homology \cite{CHSW08}, the algebraic $K$-theory \cite{KST17}, and the topological Hochschild homology \cite{BM12}.

The purpose of this paper is to study the behavior of Dolbeault cohomologies under blow-up along a smooth center.
More precisely, we prove a blow-up formula for Dolbeault cohomologies of compact complex manifolds by Cordero's Hirsch Lemma  \cite[Lemma 18]{CFGU00} and introducing relative Dolbeault cohomology in Subsection \ref{relative}.
\begin{mthm}\label{main-thm}
Let $X$ be a compact complex manifold with $\emph{dim}_{\mathbb{C}}\,X=n$ and $Z\subseteq X$ a closed complex submanifold of complex codimension $r\geq2$.
Suppose that $\pi:\tilde{X}\rightarrow X$ is the blow-up of  $X$ along $Z$.
Then for any $0\leq p,q\leq n$, there is an isomorphism
\begin{equation}\label{buf}
  H^{p,q}_{\bar{\partial}}(\tilde{X}) \cong  H^{p,q}_{\bar{\partial}}(X)\oplus \Big(\bigoplus_{i=1}^{r-1} H^{p-i,q-i}_{\bar{\partial}}(Z) \Big).
\end{equation}

\end{mthm}

As a byproduct, we obtain the isomorphism of relative Dolbeault cohomologies
$
H^{p,q}_{\bar{\partial}}(X,Z)\cong H^{p,q}_{\bar{\partial}}(\tilde{X},E)
$
with the exceptional divisor $E$, induced by the blow-up morphism $\pi:\tilde{X}\rightarrow X$
in Proposition \ref{tech-prop}.

This paper is much motivated by an interesting question in \cite[Introduction]{a}.
\begin{quest}\label{quest}
If $X$ is a $\partial\bar\partial$-manifold, is its modification $\tilde{X}$ still a $\partial\bar\partial$-manifold?
\end{quest}
Recall that a compact complex manifold is a \emph{$\partial\bar\partial$-manifold} if the standard \emph{$\partial\bar\partial$-lemma} holds on it, that is, for every pure-type $d$-closed form on this manifold, the properties of
$d$-exactness, $\partial$-exactness, $\bar{\partial}$-exactness and $\partial\bar{\partial}$-exactness are equivalent.
The converse of this question is confirmed by \cite{par} or \cite[Theorem 5.22]{DGMS}. Here we can answer Question \ref{quest} positively in the threefold case by the blow-up formulae \eqref{buf}, \eqref{1} of Dolbeault and de Rham cohomologies and an equivalent characterization of $\partial\bar\partial$-manifold as in \cite{astt}.

From the bimeromorphic geometric point of view, a blow-up transformation is a canonical and most important example of bimeromorphic map.
Conversely, we have the celebrated weak factorization theorem, (a part of) which is to be used in many occasions of this paper.
\begin{thm}[{\cite[Theorem 0.3.1]{akmw} and \cite{wl}}]\label{wft}
Let $\pi:\tilde{X}\dashrightarrow X$ be a bimeromorphic map between two compact complex manifolds  as in Definition \ref{bimero}.
Let $U$ be an open set where $\pi$ is an isomorphism.
Then $\pi$ can be factored into a sequence of blow-ups and blow-downs along irreducible nonsingular centers disjoint from $U$.
That is, to any such $\pi$ we associate a diagram
\begin{equation}\label{wf}
\pi:\tilde{X}=X_0\stackrel{\pi_1}{\dashrightarrow} X_1\stackrel{\pi_2}{\dashrightarrow}\cdots\stackrel{\pi_{i-1}}{\dashrightarrow} X_{i-1}\stackrel{\pi_{i}}{\dashrightarrow} X_{i}\stackrel{\pi_{i+1}}{\dashrightarrow}\cdots\stackrel{\pi_{l}}{\dashrightarrow} X_l=X,
\end{equation}
where
\begin{enumerate}
    \item[$(i)$]
              $\pi=\pi_l\circ\cdots\circ\pi_1;$
    \item [$(ii)$]
              $\pi_i$ are isomorphisms on $U$;
    \item [$(iii)$]
either $\pi_i:X_{i-1}\dashrightarrow X_{i}$ or $\pi_i^{-1}:X_{i}\dashrightarrow X_{i-1}$ is a morphism obtained by blowing up a nonsingular center disjoint from $U$.
 \end{enumerate}
\end{thm}

As a direct application of Theorem \ref{main-thm}, we have
\begin{cor}\label{app1}
Let $X$ and $\tilde{X}$ be two bimeromorphically equivalent $n$-dimensional compact complex manifolds.
Then for $0\leq p,q\leq n$, there hold the Dolbeault cohomology isomorphisms
\begin{equation}\label{iso-p00q'}
H^{p,0}_{\bar{\partial}}(\tilde{X}) \cong  H^{p,0}_{\bar{\partial}}(X),\quad H^{0,q}_{\bar{\partial}}(\tilde{X}) \cong  H^{0,q}_{\bar{\partial}}(X).
\end{equation}
\end{cor}
Recall that the \emph{$(p,q)$-Hodge number} $h^{p,q}(M)$ of a compact complex manifold $M$ is the complex dimension of the $(p,q)$-Dolbeault cohomology group.
Corollary \ref{app1} implies the equalities for $(p,0)$- and $(0,q)$-Hodge numbers of two bimeromorphically equivalent compact complex manifolds:
$$
h^{p,0}(\tilde{X})=h^{p,0}(X),\ h^{0,q}(\tilde{X})=h^{0,q}(X).
$$
Therefore, one obtains a uniform proof of the classical result that the $(p,0)$- and $(0,q)$-Hodge numbers of a compact complex manifold are bimeromorphic invariants.
For the type $(0,q)$, it has been shown that
$$\label{hodge-sym1}
H^{0,q}_{\bar{\partial}}(\tilde{X}) \cong  H^{0,q}_{\bar{\partial}}(X)
$$
by means of Leray spectral sequence associated with the bimeromorphic map between the two complex manifolds and the structure sheaf $\mathcal{O}_X$ (see \cite[Corollary 2.15]{u}, \cite[Proof of Corollary 1.8]{PU14}, or \cite[\S $4$ of Chapter $1$]{popa} for example).
But for the type $(p,0)$, one needs to resort to Hartogs extension theorem, such as \cite[Proposition 1.2]{u0}, or \cite[Proposition $4.1$ of Chapter $1$]{popa}.

One more new corollary of Main Theorem \ref{main-thm} is
\begin{cor}\label{enbb}
Let $\pi:\tilde{X}\dashrightarrow X$ be a bimeromorphic map between two compact complex manifolds with the weak factorization \eqref{wf}.
Then there holds the equality
$$h^{1,1}(\tilde{X})-h^{1,1}(X)=\sharp\text{\{blow-ups in \eqref{wf}\}}-\sharp\text{\{blow-downs in \eqref{wf}\}}.$$
\end{cor}
So it is obvious that if the equality $h^{1,1}(\tilde X)=h^{1,1}(X)$ for Hodge numbers holds,
then the numbers of blow-ups and blow-downs in the weak factorization \eqref{wf} are equal.

Inspired by Corollary \ref{enbb}, one obtains the bimeromorphic stability for the degeneracy of the Fr\"olicher spectral sequences at $E_1$ on compact complex threefolds and fourfolds.
\begin{thm}\label{bsf} Let $\tilde{X}$ be the blow-up of  $X$ along a smooth center $Z$.
Then the Fr\"olicher spectral sequence of $\tilde X$ degenerates at $E_1$, if and only if so do those of $X$ and $Z$.
In particular, if $\tilde{X}$ and $X$ are two bimeromorphically equivalent compact complex manifolds of dimensions at most four,
then the Fr\"olicher spectral sequence of $\tilde{X}$ degenerates at $E_{1}$ if and only if so does for $X$.
\end{thm}

Actually, the first part of Theorem \ref{bsf} is also applicable to the $\partial\bar\partial$-lemma.
Moreover, inspired by the blow-up formula of various cyclic homologies (\cite[Remark 2.11]{CHSW08}),
we obtain a blow-up formula of Hochschild homologies for compact complex manifolds by Theorem \ref{main-thm}.
Recall that the \emph{Hochschild homology} of a compact complex manifold $X$ is given by
$$
\mathrm{HH}_{k}(X):=\mathrm{Hom}_{X\times X}(\Delta^{!}\mathcal{O}_{X}[k], \mathcal{O}_{\Delta}),
$$
where $\mathcal{O}_{\Delta}$ is the structure sheaf
of the diagonal embedding $\Delta: X \to X\times X$,
and $\Delta^{!}$ is the left adjoint of the pullback functor $\Delta^{\ast}$ (cf. \cite{Cal05}).

\begin{cor}\label{hochschild}
Let $X$ be a compact complex manifold with $\emph{dim}_{\mathbb{C}}\,X=n$ and $Z\subseteq X$ a closed complex submanifold of complex codimension $r\geq2$.
Suppose that $\pi:\tilde{X}\rightarrow X$ is the blow-up of  $X$ along $Z$.
Then there is an isomorphism of Hochschild homologies
$$
\mathrm{HH}_{k}(\tilde{X}) \cong  \mathrm{HH}_{k}(X) \oplus \Big( \mathrm{HH}_{k}(Z) \Big)^{\oplus (r-1)}
$$
for any $-n\leq k\leq n$.
\end{cor}

Finally, we will list several examples with equal $(1,1)$-Hodge number or equivalently second Betti number for Corollary \ref{enbb} in Theorem \ref{exthm}, such as two bimeromorphic minimal models in birational geometry.
In particular, using the recent works of Graf \cite{g} and Lin \cite{lin1,lin2} on algebraic approximation of K\"ahler threefolds,
one obtains a bimeromorphic invariance result of Hodge numbers.
\begin{prop}[=Proposition \ref{lin3}]\label{0lin3}
Let ${X}$ and $\tilde X$ be two bimeromorphic K\"ahler minimal models of dimension three with nonnegative Kodaira dimension except two.
Then they have the same Hodge numbers.
\end{prop}

According to the blow-up formulae of de Rham cohomology \eqref{1} and Dolbeault cohomology for general complex manifolds,
it is reasonable to propose the following conjecture for that of Bott-Chern cohomology for general complex manifolds.
\begin{conjecture}\label{conjbc}
Let $X$ be a compact complex manifold with $\emph{dim}_{\mathbb{C}}\,X=n$ and $Z\subseteq X$ a closed complex submanifold of complex codimension $r\geq2$.
Suppose that $\pi:\tilde{X}\rightarrow X$ is the blow-up of  $X$ along $Z$.
Then there is a canonical isomorphism
$$
H^{p,q}_{BC}(\tilde{X}) \cong  H^{p,q}_{BC}(X)\oplus \Big(\bigoplus_{i=1}^{r-1} H^{p-i,q-i}_{BC}(Z) \Big),
$$
for any $0\leq p,q\leq n$.
\end{conjecture}
Here the \emph{$(p,q)$-Bott-Chern cohomology group} of a complex manifold $M$ is defined by
$$H^{p,q}_{BC}(M):=\frac{\ker \partial\cap \ker\bar{\partial}}{\im\ \partial\bar{\partial}}.$$
 It is worth noticing that if $\tilde{X}$ satisfies the $\partial\bar\partial$-lemma, so does $Z$ by Theorem \ref{main-thm} and \eqref{1}, and thus Conjecture \ref{conjbc} holds then.

This paper is organized as follows.
We devote Section \ref{pre} to review the blow-up formula of de Rham cohomologies on complex manifolds and introduce relative Dolbeault cohomology.
In Section \ref{pmt}, we give the proof of Main Theorem \ref{main-thm}.
The proofs of Corollaries \ref{app1}-\ref{hochschild} of Main Theorem \ref{main-thm} are given in Section \ref{amt}.
In Section \ref{ex}, we list several examples specially for Corollary \ref{enbb}. The final appendix \ref{appendix} is to give a new proof for blow-up formula of de Rham cohomologies on complex manifolds by relative de Rham cohomologies.

Shortly after we posted our first version \cite{ryy}$_{v1}$ on arXiv, we were informed that D. Angella, T. Suwa, N. Tardini and A. Tomassini also obtained a similar result \cite[Theorem 2.1]{astt} to Theorem \ref{bsf} with the center admitting a holomorphically contractible neighbourhood by \v{C}ech cohomology theory, and additionally considered the orbifold case for new
\cite[Examples 3.2 and 3.3]{astt} satisfying the $\partial\bar\partial$-lemma. We also notice the more recent works \cite{meng1,meng2} of L. Meng, which present explicit expression for the isomorphism in the blow-up formula \eqref{buf}, and J. Stelzig's important work \cite{Jo18} which proves a similar result to \eqref{buf} by computing double complexes of blowing up complex manifolds up to a suitable quasi-isomorphism and provides us a critical equivalent isomorphism to that of Proposition \ref{tech-prop}. More recently, in his updating of \cite{meng1}, Meng proves the vanishing of the direct image sheaves relating to the relative Dolbeault sheaves and thus one is still able to obtain Proposition \ref{tech-prop} by our previous approach in \cite{ryy}$_{v3}$, which is sketched in Remark \ref{meng-rk}.

\section{Preliminaries}\label{pre}

In this section we will recall a blow-up formula for de Rham cohomology and  introduce relative Dolbeault cohomology, which plays an important role in the proof of Main Theorem \ref{main-thm}.
\subsection{Blow-up formula for de Rham cohomology}

Assume that $X$ is a compact complex manifold in the \emph{Fujiki Class ($\mathscr{C}$)}, i.e., bimeromorphic to a K\"ahler manifold.
Let $\textmd{dim}_{\mathbb{C}}X=n$ and let $Z\subseteq X$ be a closed complex submanifold with $\textmd{codim}_{\mathbb{C}}Z=r\geq 2$.
Then $Z$ is also in the Fujiki Class ($\mathscr{C}$) (cf. \cite[Lemma 4.6]{Fu78}).
Let $\pi:\tilde{X}\rightarrow X$ be the blow-up of $X$ with the center $Z$ and the exceptional divisor $E=\pi^{-1}(Z)$.
By definition, we get that $\tilde{X}$ is in the Fujiki Class ($\mathscr{C}$).
Observe that $E$ is a hypersurface in $\tilde{X}$ the inclusion $\jmath:E\hookrightarrow\tilde{X}$ induces a map called {Gysin morphism}
$$
\jmath_{*}:H^{\bullet}_{dR}(E;\mathbb{C})\rightarrow H^{\bullet+2}_{dR}(\tilde{X};\mathbb{C}).
$$
In particular, we have the following canonical isomorphism which gives rise to a blow-up formula for de Rham cohomology
\footnote{In the proof of \cite[Theorem 7.31]{V} the manifold $X$ is a K\"{a}hler manifold for the studying of Hodge structure of a blow-up;
in fact, the argument given in the proof of \cite[Theorem 7.31]{V} can be applied to any compact complex manifold without any essential changes.}
(cf. \cite[Theorem 7.31]{V}):
\begin{equation}\label{1}
\xymatrix@C=0.5cm{
 & H^{k}_{dR}(X;\mathbb{C})\oplus\biggl(\bigoplus^{r-2}_{i=0}H^{k-2i-2}_{dR}(Z;\mathbb{C})\biggr)
 \ar[r]^{\quad\quad\qquad\qquad\phi} & H^{k}_{dR}(\tilde{X};\mathbb{C})& }
\end{equation}
where $\phi=\pi^{*}+\sum^{r-2}_{i=0}\jmath_{*}\circ \bm h^{i}\circ (\pi_{E})^{*}$, $\bm h=c_{1}(\mathcal{O}_{E}(1))\in H^{2}_{dR}(E;\mathbb{R})$
and $\bm h^{i}$ is given by the cup-product by $\bm h^{i}\in H^{2i}_{dR}(E;\mathbb{R})$.
We will present a new proof of this formula by relative de Rham cohomologies in Appendix \ref{appendix}.
By definition, $E$ is biholomorphic to $\mathbb{P}(\mathcal{N}_{Z/X})$,
the projective bundle associated to the normal bundle $\mathcal{N}_{Z/X}$,
and $\mathcal{O}_{E}(1)$ is the associated tautological line bundle.
Recall that the \emph{normal bundle} $\mathcal{N}_{Z/X}:=T_{X|Z}/T_Z$ of $Z$ in $X$ is a holomorphic vector bundle of rank $r$.
Moreover, the restriction of $\bm h$ to each fiber $\mathbb{CP}^{r-1}$ of $\mathbb{P}(\mathcal{N}_{Z/X})$ is the generator of the cohomology ring $H^{\bullet}_{dR}(\mathbb{CP}^{r-1},\mathbb{R})$, which means
$$
\bm h|_{\mathbb{CP}^{r-1}}\in H^{1,1}_{\bar{\partial}}(\mathbb{CP}^{r-1}).
$$
This implies that $\bm h\in H^{2}_{dR}(E;\mathbb{R})\cap H^{1,1}_{\bar{\partial}}(E)$ and hence
$\bm h^{i}\in H^{2i}_{dR}(E;\mathbb{R})\cap H^{i,i}_{\bar{\partial}}(E)$.
Since every compact complex manifold in the Fujiki class ($\mathscr{C}$) admits a strong Hodge decomposition (cf. \cite[Theorem 12.9]{Dem12})
we have the canonical decompositions
\begin{eqnarray}
H^{k}_{dR}(X;\mathbb{C})&\cong& \bigoplus_{p+q=k}H^{p,q}_{\bar{\partial}}(X),\,\,\,0\leq k\leq 2n;
\label{2} \\
H^{k}_{dR}(\tilde{X};\mathbb{C})&\cong& \bigoplus_{p+q=k}H^{p,q}_{\bar{\partial}}(\tilde{X}),\,\,\,0\leq k\leq 2n;\label{3}\\
H^{l}_{dR}(Z;\mathbb{C}) &\cong& \bigoplus_{s+t=l}H^{s,t}_{\bar{\partial}}(Z),\,\,\,0\leq l\leq 2(n-r).\label{4}
\end{eqnarray}

According to the isomorphism (\ref{1}) there exist unique classes
$[\beta]\in H^{k}_{dR}(X;\mathbb{C})$ and $[\gamma]\in H^{k-2i-2}_{dR}(Z;\mathbb{C})$ for each class $[\alpha]\in H^{k}_{dR}(\tilde{X};\mathbb{C})$ such that
$$\label{5}
[\alpha]=\pi^{*}[\beta]+\sum^{r-2}_{i=0}\jmath_{*}(\bm h^{i}\wedge(\pi_{E})^{*}[\gamma]).
$$
From the Hodge decomposition (\ref{2}) it follows
\begin{equation}\label{6}
[\alpha]=\sum_{p+q=k}[\alpha]_{(p,q)},
\end{equation}
where $[\alpha]_{(p,q)}\in H^{p,q}_{\bar{\partial}}(\tilde{X})$.
Likewise, from (\ref{3}) and (\ref{4}) one obtains
\begin{equation}\label{7}
[\beta]=\sum_{p+q=k}[\beta]_{(p,q)},
\end{equation}
and
\begin{equation}\label{8}
[\gamma]=\sum_{s+t=k-2i-2}[\gamma]_{(s,t)}.
\end{equation}
From (\ref{6})-(\ref{8}) and via a degree checking we get
$$
[\alpha]_{(p,q)}=\pi^{*}[\beta]_{(p,q)}+\sum^{r-2}_{i=0}\jmath_{*}(\bm h^{i}\wedge(\pi_{E})^{*}[\gamma]_{(p-i-1,q-i-1)}).
$$
This implies a blow-up formula for Dolbeault cohomology on complex manifolds with canonical decompositions
$$
  H^{p,q}_{\overline{\partial}}(\tilde{X})
  \cong H^{p,q}_{\bar{\partial}}(X)\oplus\biggl(\bigoplus^{r-2}_{i=0}H^{p-i-1,q-i-1}_{\bar{\partial}}(Z)\biggr)
  = H^{p,q}_{\bar{\partial}}(X)\oplus\biggl(\bigoplus^{r-1}_{i=1}H^{p-i,q-i}_{\bar{\partial}}(Z)\biggr),
$$
where $0\leq p,q\leq n$.

The goal of this paper is to present a blow-up formula for Dolbeault cohomologies of a general compact complex manifold.
It is worth noticing that not all complex manifolds satisfy the $\partial\bar\partial$-lemma or rather Hodge decomposition, such as Iwasawa manifolds.
%

\subsection{The exact sequence associated to a closed submanifold}\label{relative}
In this subsection we introduce the definition of relative Dolbeault cohomology inspired by the relative de Rham cohomology in the sense of Godbillon \cite[Chapitre XII]{Go97}.
Another version of relative Dolbeault cohomology was defined by Suwa \cite{Su17}.

%

Let $X$ be a compact complex manifold of complex dimension $n$.
For $0\leq p\leq n$, there is a complex of complex-valued differential forms
$$
0 \rightarrow \mathcal{A}^{p,0}(X) \xrightarrow{\bar{\partial}}  \mathcal{A}^{p,1}(X)\xrightarrow{\bar{\partial}} \cdots \xrightarrow{\bar{\partial}} \mathcal{A}^{p,n}(X)\xrightarrow{\bar{\partial}}  0,
$$
where $\mathcal{A}^{p,q}(X)$ is the space of complex-valued differential forms of $(p,q)$-type on $X$.
Then \emph{the $(p,q)$-th Dolbeault cohomology} of $X$ is defined to be
\begin{equation*}
H_{\bar{\partial}}^{p,q}(X):=\frac{\ker\,(\bar{\partial}:\mathcal{A}^{p,q}(X)\rightarrow \mathcal{A}^{p,q+1} (X))}{\mathrm{Im}\,(\bar{\partial}:\mathcal{A}^{p,q-1}(X) \rightarrow \mathcal{A}^{p,q} (X))}.
\end{equation*}



Assume that $M$ is a compact complex manifold with complex dimension $n$ and let $N$ be a closed complex submanifold of $M$.
For any $p\geq0$, consider the space of differential forms
$$
\mathcal{A}^{p,\bullet}(M,N)
=\{\alpha\in\mathcal{A}^{p,\bullet}(M)\,|\,i^{\ast}\alpha=0\},
$$
where $i^{\ast}$ is the pullback of the holomorphic inclusion $i:N\hookrightarrow M$.
Since $\mathcal{A}^{p,\bullet}(M,N)$ is closed under the action of the differential operator $\bar{\partial}$ we get a sub-complex of the Dolbeault complex $\{\mathcal{A}^{p,\bullet}(M),\bar{\partial}\}$, called the \emph{relative Dolbeault complex}, with respect to $N$:
$$
\xymatrix{
0 \ar[r]^{} & \mathcal{A}^{p,0}(M,N) \ar[r]^{\bar{\partial}} & \mathcal{A}^{p,1}(M,N)  \ar[r]^{\bar{\partial}} & \mathcal{A}^{p,2}(M,N) \ar[r]^{\quad\bar{\partial}} & \cdots.}
$$
The associated $q$-th cohomology $H^{p,q}_{\bar{\partial}}(M,N)$ is called the \emph{relative Dolbeault cohomology group} of the pair $(M,N)$.
From definition, it is straightforward to verify that if $p>\mathrm{dim}_{\mathbb{C}}\,N$ or $q>\mathrm{dim}_{\mathbb{C}}\,N$ then $H^{p,q}_{\bar{\partial}}(M,N)=H^{p,q}_{\bar{\partial}}(M)$.

\begin{lem}\label{sur-1}
There exists an open neighborhood $\mathcal{U}$ of $N$ in $M$ such that $i^{\ast}:\mathcal{A}^{p,q}(\mathcal{U})\rightarrow\mathcal{A}^{p,q}(N)$ is surjective,
where $i:N\hookrightarrow \mathcal{U}$ is the inclusion.
\end{lem}

\begin{proof}
By the classical Tubular Neighborhood Theorem \cite[Theorem 6.24]{lee}, we have an open tubular neighborhood $\mathcal{U}$ of $N$ with a {smooth} retraction map $\gamma:\mathcal{U}\rightarrow N$ such that $\gamma|_{N}$ is the identity map of $N$ as in \cite[Proposition 6.25]{lee}.
Now, to $\mathcal{U}$ and $N$ we associate the induced complex structures by the one of $M$, and let $i:Z\hookrightarrow \mathcal{U}$ be the holomorphic embedding since $Z$ is a closed complex submanifold of $M$. Then the chain rule ensures that $\gamma$ is still smooth under these complex structures.

For any $\alpha^{p,q}\in\mathcal{A}^{p,q}(N)$, the pull-back $\beta:=\gamma^{\ast}(\alpha^{p,q})$ by $\gamma$ is a complex-valued {smooth} $(p+q)$-form on $\mathcal{U}$.
By the type decomposition according to the complex structure,
$\beta$ has the unique decomposition
$
\beta=\sum_{s+t=p+q}\beta^{s,t}
$
with $\beta^{s, t}\in\mathcal{A}^{s,t}(\mathcal{U})$.
So,
$$
\alpha^{p, q}=i^{\ast}(\gamma^{\ast}(\alpha^{p,q}))=i^{\ast}(\beta)=\sum_{s+t=p+q}i^{\ast}(\beta^{s,t})=i^{\ast}(\beta^{p,q})
$$
since $\gamma\circ i=\textmd{id}_{N}$ and the pull-back of the holomorphic map $i$ preserves the pure types of complex differential forms. Hence, $\beta^{p,q}$ is the desired one.
\end{proof}

As a direct corollary of Lemma \ref{sur-1}, one has
\begin{lem}
The pullback $i^{\ast}:\mathcal{A}^{p,q}({M})\rightarrow \mathcal{A}^{p,q}({N})$ is surjective.
\end{lem}

Particularly, there holds the so-called \emph{short exact sequence for the pair $(M,N)$} of complexes
$$\label{short-exact-relative}
\xymatrix@C=0.5cm{
0 \ar[r]^{} & \mathcal{A}^{p,\bullet}(M,N) \ar[r]^{} & \mathcal{A}^{p,\bullet}(M)  \ar[r]^{i^{\ast}} & \mathcal{A}^{p,\bullet}(N) \ar[r]^{} & 0},
$$
which gives rise to a long exact sequence
\begin{equation}\label{long-exact-relative}
\xymatrix@C=0.5cm{
\cdots\ar[r]^{} & H^{p,\bullet}_{\bar{\partial}}(M,N) \ar[r]^{} & H^{p,\bullet}_{\bar{\partial}}(M)  \ar[r]^{} & H^{p,\bullet}_{\bar{\partial}}(N) \ar[r]^{} & H^{p,\bullet+1}_{\bar{\partial}}(M,N) \ar[r]^{} & \cdots}.
\end{equation}

\section{Proof of Main Theorem \ref{main-thm}}\label{pmt}

\subsection{Dolbeault cohomology of projective bundles}
We first recall Hirsch Lemma for Dolbeault cohomology in \cite[Section 4.2]{CFGU00} that will provide a model for the Dolbeault cohomology of the total space in a holomorphic fibration under some suitable hypothesis.

Let $F\hookrightarrow E\xrightarrow{\pi} B$ be a holomorphic fibration,
where $E,B,F$ are compact connected complex manifolds and the structure group of the fibration is connected.
An element $\bm\alpha\in H^{p,q}_{\bar\partial}(F)$ is called \emph{transgressive}
if there is a representative $\alpha\in \mathcal{A}^{p,q}(F)$
which extends to a form $\tilde{\alpha}\in \mathcal{A}^{p,q}(E)$ such that $\bar\partial\tilde{\alpha}=\pi^{\ast}\beta$ for some $\bar\partial$-closed form $\beta\in \mathcal{A}^{p,q+1}(B)$.
If $H^{\bullet,\bullet}_{\bar\partial}(F)$ is free as a bigraded algebra,
we say that it is \emph{transgressive} if it has an algebra basis consisting of transgressive elements.

Assume that the bigraded algebra $H^{\bullet,\bullet}_{\bar\partial}(F)$ for the holomorphic fibration is free and transgressive.
Let $(A^{\bullet,\bullet},\bar\partial)$ be a differential bigraded algebra and
$$
\rho:A^{\bullet,\bullet}\rightarrow \mathcal{A}^{\bullet,\bullet}(B)
$$
a morphism of differential bigraded algebras giving an isomorphism on cohomology;
that is, $(A^{\bullet,\bullet},\bar\partial)$ is a model for $(\mathcal{A}^{\bullet,\bullet}(B),\bar\partial)$.
Pick an algebra basis $\{\bm x_1,\cdots,\bm x_p\}$ for $H^{\bullet,\bullet}_{\bar\partial}(F)$.
Let  $\tilde{\alpha}_i\in \mathcal{A}^{p,q}(E)$ be a form and give rise to a $\bar\partial$-closed form representing $\bm x_i$ when restricted to $F$.
Let $\beta_i$ be such that  $\bar\partial\tilde{\alpha}_i=\pi^{\ast}\beta_i$.
Since $\rho$ is an isomorphism on cohomology,
one may pick $a_i$ such that $\beta_i=\rho(a_i)$ for some $\bar\partial$-closed form $a_i\in A^{\bullet,\bullet}$.
Let $$T=A^{\bullet,\bullet}\otimes H^{\bullet,\bullet}_{\bar\partial}(F)$$ be the tensor product of these bigraded algebras,
and define a differential $\bar\partial$ of type $(0,1)$ for $T$ by setting $$\bar\partial: H^{\bullet,\bullet}_{\bar\partial}(F)\rightarrow A^{\bullet,\bullet+1}$$ as $\bar\partial(\bm x_i)=a_i$.
Then $(T,\bar\partial)$ is a differential bigraded algebra.
Then one has the important Hirsch Lemma.
\begin{lem}[{\cite[Lemma 18]{CFGU00}}]\label{hirsch}
The morphism
$$
\tilde{\rho}:T=A^{\bullet,\bullet}\otimes H^{\bullet,\bullet}_{\bar\partial}(F)\rightarrow \mathcal{A}^{\bullet,\bullet}(E),
$$
defined by $\tilde{\rho}|_A=\pi^{\ast}\circ\rho$ and $\tilde{\rho}(\bm x_i)=\tilde{\alpha}_i$,
induces an isomorphism on cohomology.
Hence, $(A^{\bullet,\bullet}\otimes H^{\bullet,\bullet}_{\bar\partial}(F),\bar\partial)$ is a model for the Dolbeault complex $(\mathcal{A}^{\bullet,\bullet}(E),\bar\partial)$.
\end{lem}

Now we apply the Hirsch Lemma to projective bundles.
Suppose that $V$ is a holomorphic vector bundle of rank $r$ over a connected compact complex manifold $B$ of $\textmd{dim}_{\mathbb{C}}B=n$.
Consider the projectivization of the bundle $V$.
Then we get a holomorphic fiber bundle $\mathbb{P}(V)$ over $B$ with the fibre $F\cong \mathbb{CP}^{r-1}$.
Note that the total space $\mathbb{P}(V)$ is connected and the structure group $\mathrm{PGL(r,\mathbb{C})}$ is also connected.
One has to check the conditions in Hirsch Lemma for $\mathbb{P}(V)$.

\begin{lem}\label{transg}
The Dolbeault cohomology ring $H_{\bar{\partial}}^{\bullet,\bullet}(F)$ is a transgressive free bialgebra over $\mathbb{C}$.
\end{lem}

\begin{proof}
Note that $F\cong\mathbb{CP}^{r-1}$.
Let $t$ be the K\"{a}hler form of Fubini-Study metric on $\mathbb{CP}^{r-1}$.
Then the de Rham cohomology ring of $\mathbb{CP}^{r-1}$ is $\mathbb{C}[\bm t]/(\bm t^{r})$.
On the other hand, the generator $\bm t$ can be thought of as an element of $H^{1,1}_{\bar{\partial}}(\mathbb{CP}^{r-1})$.
According to the Hodge decomposition theorem we get that the Dolbeault cohomology ring of $\mathbb{CP}^{r-1}$ is isomorphic to $\mathbb{C}[\bm t]/(\bm t^{r})$ as a free bialgebra over $\mathbb{C}$.
Therefore, it remains to show that the generator $\bm t$ is transgressive.

Set $\mathcal{O}_{\mathbb{P}(V)}(1)$ as the tautological line bundle over $\mathbb{P}(V)$ and $\tilde{\bm t}=c_1(\mathcal{O}_{\mathbb{P}(V)}(1))$.
Then $\tilde{\bm t}$ represents a non-trivial class in $H_{\bar{\partial}}^{1,1}(\mathbb{P}(V))$ such that the restriction of
$\tilde{\bm t}$ is just the generator of $H^{\bullet,\bullet}_{\bar{\partial}}(F)$, i.e., $\tilde{\bm t}|_{F}=\bm t$.
Since $ \bar{\partial}\tilde{\bm t}=0$, we have
$\bar{\partial}\tilde{\bm t}=\pi^{\ast} \beta,$
where $\beta$ is the zero form on $B$.
From definition, we get that $\bm t$ is transgressive and hence each generator of $H_{\bar{\partial}}^{\bullet,\bullet}(F)$ is transgressive.
This completes the proof.
\end{proof}

In particular, as a direct consequence of Lemma \ref{transg} and the Hirsch Lemma \ref{hirsch} we have the following result.
\begin{prop}\label{projective-formula}
For any $0\leq p, q\leq n$, we have the following identity:
\begin{equation*}
H_{\bar{\partial}}^{p,q}(\mathbb{P}(V))
=\bigoplus_{i=0}^{r-1} \tilde{\bm t}^{i}\wedge\pi^{\ast}H_{\bar{\partial}}^{p-i,q-i}(B),
\end{equation*}
where $\tilde{\bm t}=c_1(\mathcal{O}_{\mathbb{P}(V)}(1))$.
\end{prop}

\begin{proof}
Consider the projective bundle
$$
\pi: \mathbb{P}(V)\to B.
$$
Let $A^{\bullet,\bullet}:=\mathcal{A}^{\bullet,\bullet}(B)$ and let $\rho$ be the identity map from $A^{\bullet,\bullet}$ to $\mathcal{A}^{\bullet,\bullet}(B)$.
Consider the tensor product $T:=A^{\bullet,\bullet}\otimes H^{\bullet,\bullet}_{\bar{\partial}}(F)$.
From the proof of Lemma \ref{transg} there exists a natural differential $\bar{\partial}_{T}$ of type (0,1) by
$
\bar{\partial}_{T}(a\otimes b)=(\bar{\partial}a)\otimes b,
$
for any $a\otimes b\in T$.
In addition, we can define a bialgebra morphism
\begin{equation}\label{equ3.1}
\tilde{\rho}:T=A^{\bullet,\bullet}\otimes H^{\bullet,\bullet}_{\bar{\partial}}(F)
\rightarrow \mathcal{A}^{\bullet,\bullet}(\mathbb{P}(V))
\end{equation}
by setting $\tilde{\rho}|_{A^{\bullet,\bullet}}=\pi^{\ast}$ and $\tilde{\rho}(\bm t)=\tilde{\bm t}$.
On the one hand, according to the Hirsch Lemma \ref{hirsch} we get that the map $\tilde{\rho}$ in (\ref{equ3.1}) induces an isomorphism on cohomology:
\begin{equation}\label{equ3.2}
   \tilde{\rho}:H^{\bullet,\bullet}_{\bar{\partial}_{T}}(A^{\bullet,\bullet}\otimes H^{\bullet,\bullet}_{\bar{\partial}}(F))
   \stackrel{\cong}{\longrightarrow}  H^{\bullet,\bullet}_{\bar{\partial}}(\mathbb{P}(V)).
\end{equation}
On the other hand, from the definition of the operator $\bar{\partial}_{T}$ we get
$$\label{equ3.3}
H^{\bullet,\bullet}_{\bar{\partial}_{T}}(A^{\bullet,\bullet}\otimes H^{\bullet,\bullet}_{\bar{\partial}}(F))
=H^{\bullet,\bullet}_{\bar{\partial}}(B)\otimes H^{\bullet,\bullet}_{\bar{\partial}}(F).
$$
This implies that the isomorphism in (\ref{equ3.2}) is equivalent to
\begin{equation}\label{equ3.4}
\tilde{\rho}: H_{\bar{\partial}}^{\bullet,\bullet}(B) \otimes H^{\bullet,\bullet}_{\bar{\partial}}(F)
\stackrel{\cong}{\longrightarrow} H_{\bar{\partial}}^{\bullet,\bullet}(\mathbb{P}(V)).
\end{equation}
Via a degree checking in (\ref{equ3.4}) we have the following identity:
$$
H^{p,q}_{\bar{\partial}}(\mathbb{P}(V))
=  \tilde{\rho}\biggl(\sum^{r-1}_{i=0}\bm t^{i}\wedge H^{p-i,q-i}_{\bar{\partial}}(B)\biggr)
= \sum^{r-1}_{i=0}\tilde{\bm t}^{i}\wedge \pi^{*}H^{p-i,q-i}_{\bar{\partial}}(B).
$$
This completes the proof.
\end{proof}
\subsection{Dolbeault blow-up formula}\label{sec-bu}
From now on we assume that $X$ is a compact complex manifold of complex dimension $n$.
Suppose that $\imath:Z \hookrightarrow X$ is a closed complex submanifold of complex codimension $r\geq2$.
Without loss of generality, we assume that $Z$ is connected;
otherwise, we can carry out the blow-up operation along each connected component of $Z$ step by step.
Recall that the \emph{normal bundle} $T_{X|Z}/T_Z$ of $Z$ in $X$, denoted by $\mathcal{N}_{Z/X}$, is a holomorphic vector bundle of rank $r$.
The \emph{blow-up $\tilde{X}$ of $X$ with center $Z$} is a projective morphism  $\pi: \tilde{X}\to X$ such that
$$
\pi: \tilde{X}-E \to X-Z
$$
is a biholomorphism.
Here
$$
E:=\pi^{-1}(Z)\cong \mathbb{P}(\mathcal{N}_{Z/X})
$$
is the \emph{exceptional divisor} of the blow-up.
Then one has the following blow-up diagram
\begin{equation}\label{blow-up-diag3}
\xymatrix{
E \ar[d]_{\pi_{E}} \ar@{^{(}->}[r]^{\tilde{\imath}} & \tilde{X}\ar[d]^{\pi}\\
 Z \ar@{^{(}->}[r]^{\imath} & X.
}
\end{equation}
In particular, due to Proposition \ref{projective-formula}, for any $0\leq p,q\leq n-1$,
the $(p,q)$-Dolbeault cohomology of the exceptional divisor $E$ is
\begin{equation}\label{equ3.5}
H_{\bar{\partial}}^{p,q}(E)
\cong \bigoplus_{i=0}^{r-1} \tilde{\bm t}^{i}\wedge\pi_{E}^{\ast}H_{\bar{\partial}}^{p-i,q-i}(Z),
\end{equation}
where $\tilde{\bm t}=c_1(\mathcal{O}_{E}(1))$.


As the blow-up morphism $\pi$ induces a natural commutative diagram for the short exact sequences of complexes
\begin{equation*}
\xymatrix@C=0.5cm{
0 \ar[r]^{} & \mathcal{A}^{p,\bullet}(X,Z) \ar[d]_{\pi^{\ast}} \ar[r]^{} & \mathcal{A}^{p,\bullet}(X)\ar[d]_{\pi^{\ast}} \ar[r]^{} & \mathcal{A}^{p,\bullet}(Z)
\ar[d]_{\pi_{E}^{\ast}} \ar[r]^{} & 0 \\
0 \ar[r] & \mathcal{A}^{p,\bullet}(\tilde{X},E) \ar[r]^{} & \mathcal{A}^{p,\bullet}(\tilde{X})
 \ar[r]^{} &
\mathcal{A}^{p,\bullet}(E)\ar[r]^{} & 0,}
\end{equation*}
the long exact sequence \eqref{long-exact-relative} and the standard diagram chasing give rise to a commutative diagram
$$
\xymatrix@C=0.5cm{
\cdots \ar[r]^{} & H_{\bar{\partial}}^{p, q}(X,Z)\ar[d]_{\pi^{\ast}} \ar[r]^{}
&H_{\bar{\partial}}^{p,q}(X)
\ar[d]_{\pi^{\ast}} \ar[r]^{}
& H^{p,q}_{\bar{\partial}}(Z)
\ar[d]_{\pi_{E}^{\ast}}
\ar[r]^(0.4){\delta}
& H_{\bar{\partial}}^{p, q+1}(X,Z) \ar[d]_{\pi^{\ast}} \ar[r]^{}&\cdots\\
\cdots \ar[r]^{} & H_{\bar{\partial}}^{p, q}(\tilde{X},E) \ar[r]^{} & H_{\bar{\partial}}^{p,q}(\tilde{X}) \ar[r]^{}
& H^{p,q}_{\bar{\partial}}(E) \ar[r]^(0.4){\tilde{\delta}}
& H_{\bar{\partial}}^{p, q+1}(\tilde{X},E) \ar[r]^{}& \cdots,}
$$
where $\delta,\tilde{\delta}$ are the corresponding coboundary operators.
\begin{prop}\label{tech-prop}
The blow-up morphism $\pi:\tilde{X}\rightarrow X$ induces an isomorphism
$$
H^{p,q}_{\bar{\partial}}(X,Z)\cong H^{p,q}_{\bar{\partial}}(\tilde{X},E)
$$
for each $p,q\geq0$.
\end{prop}
Proposition \ref{tech-prop} is to be proved in Subsection \ref{prop3.7}, where the equivalence of the isomorphisms $H^{p,q}_{\bar{\partial}}(\tilde{X}|X)\cong H^{p,q}_{\bar{\partial}}(E|Z)$ and  $
H^{p,q}_{\bar{\partial}}(X,Z)\cong H^{p,q}_{\bar{\partial}}(\tilde{X},E)
$ will be given.
Since
$
\pi:\tilde{X}\rightarrow X$ and $\pi_{E}:E\rightarrow Z
$
are proper surjective holomorphic maps,
$
\pi^{\ast}:H^{p,q}_{\bar{\partial}}(X)\rightarrow H^{p,q}_{\bar{\partial}}(\tilde{X})
$
and thus
$
\pi_{E}^{\ast}:H^{p,q}_{\bar{\partial}}(Z)
\rightarrow
H^{p,q}_{\bar{\partial}}(E)
$
are injective by \cite[Theorem 3.1]{Wells74} and the Weak Five Lemma \cite[Lemma $3.3.(i)$]{mc}, respectively. Moreover, one has:
\begin{prop}[{cf. \cite[Proposition 3.3]{YZ15}}]\label{cok-is}
 Consider a commutative diagram of abelian groups such that its horizontal rows are exact
\begin{equation*}
\xymatrix@C=0.5cm{
  \cdots \ar[r]^{} & A_1 \ar[d]_{i_1} \ar[r]^{f_1}& A_2 \ar[d]_{i_2} \ar[r]^{f_2}& A_3 \ar[d]_{i_3} \ar[r]^{f_3}& A_4 \ar[d]_{i_4} \ar[r]^{f_4}& A_5 \ar[d]_{i_5} \ar[r]^{} & \cdots \\
\cdots \ar[r]^{} & B_1 \ar[r]^{g_1}& B_2 \ar[r]^{g_2}& B_3 \ar[r]^{g_3}& B_4 \ar[r]^{g_4}& B_5   \ar[r]^{} & \cdots . }
\end{equation*}
Assume that $i_1$ is epimorphic, $i_2,i_3,i_5$ are monomorphic and $i_4$ is isomorphic.
Then there exists a natural isomorphism
$$
\mathrm{coker}\,i_2\cong\mathrm{coker}\,i_3.
$$
\end{prop}

Based on these,
we have
\begin{equation}\label{pre-blow-up-formula}
H_{\bar{\partial}}^{p, q}(\tilde{X})
\cong H_{\bar{\partial}}^{p, q}(X)
\oplus \Big( H_{\bar{\partial}}^{p, q}(E)/ \pi_{E}^{\ast}H_{\bar{\partial}}^{p, q}(Z)\Big).
\end{equation}
From \eqref{equ3.5} and \eqref{pre-blow-up-formula} it follows that
$$
H_{\bar{\partial}}^{p,q}(\tilde{X})
\cong H_{\bar{\partial}}^{p,q}(X)\oplus \left(\bigoplus_{i=1}^{r-1} \tilde{\bm{t}}^{i}\wedge\pi_{E}^{\ast}H_{\bar{\partial}}^{p-i,q-i}(Z) \right)
\cong  H_{\bar{\partial}}^{p,q}(X)\oplus \left(\bigoplus_{i=1}^{r-1}H_{\bar{\partial}}^{p-i,q-i}(Z) \right).
$$
This completes the proof of Theorem \ref{main-thm}.

\subsection{Proof of Proposition \ref{tech-prop}}\label{prop3.7}
This proof is based on the recent results of J. Stelzig \cite{Jo18} and some standard homological algebra techniques.



It is easy to see that for any $p,q\geq0$ the pullback $\pi^{\ast}:\mathcal{A}^{p,q}(X)\rightarrow\mathcal{A}^{p,q}(\tilde{X})$
is injective.
In fact, for $\alpha\in \mathcal{A}^{p,q}(X)$ with $\pi^{\ast}\alpha=0$, $\alpha|_{X-Z}=0$ since
$\pi_{\tilde{X}- E}:\tilde{X}- E\rightarrow X- Z$
is biholomorphic.
Then the continuity argument by $\mathrm{codim}_{X}\,Z\geq2$ shows that $\alpha=0$. So we get an injective morphism of complexes
$$
\pi^{\ast}:\{\mathcal{A}^{\bullet,\bullet}(X),\bar{\partial}\}
\rightarrow\{\mathcal{A}^{\bullet,\bullet}(\tilde{X}),\bar{\partial}\}.
$$
Let $\mathcal{A}^{\bullet,\bullet}(\tilde{X}|X)
=\mathcal{A}^{\bullet,\bullet}(\tilde{X})/\pi^{\ast}\mathcal{A}^{\bullet,\bullet}(X)$
be the quotient complex.
Then we obtain a short exact sequence of complexes
$$\label{short-exact-1}
\xymatrix{
  0 \ar[r] & \mathcal{A}^{\bullet,\bullet}(X) \ar[r]^{\pi^{\ast}}
  &  \mathcal{A}^{\bullet,\bullet}(\tilde{X})\ar[r]^{\mathrm{pr}\ } & \mathcal{A}^{\bullet,\bullet}(\tilde{X}|X) \ar[r] & 0 }
$$
and thus the long exact sequence of cohomology groups:
$$\label{long-exact-1}
\xymatrix@C=0.5cm{
\cdots\ar[r]^{} & H^{p,q-1}_{\bar{\partial}}(\tilde{X}|X) \ar[r]^{} & H^{p,q}_{\bar{\partial}}(X)  \ar[r]^{\pi^{\ast}}
& H^{p,q}_{\bar{\partial}}(\tilde{X}) \ar[r]^{} & H^{p,q}_{\bar{\partial}}(\tilde{X}|X) \ar[r]^{} & \cdots.}
$$

Observe that $\pi_{E}:E\rightarrow Z$ is a fibre bundle and then the pullback
$\pi^{\ast}_{E}:\mathcal{A}^{p,q}(Z)\rightarrow\mathcal{A}^{p,q}(E)$ is injective.
Likewise, we have the long exact sequence of cohomology groups
$$\label{long-exact-2}
\xymatrix@C=0.5cm{
\cdots\ar[r]^{} & H^{p,q-1}_{\bar{\partial}}(E|Z) \ar[r]^{} & H^{p,q}_{\bar{\partial}}(Z)  \ar[r]^{\pi^{\ast}_{E}}
& H^{p,q}_{\bar{\partial}}(E) \ar[r]^{} & H^{p,q}_{\bar{\partial}}(E|Z) \ar[r]^{} & \cdots.}
$$
Hence, the blow-up diagram \eqref{blow-up-diag3} induces a commutative diagram
\begin{equation}\label{comm-diagram3.1}
\xymatrix@C=0.5cm{
   \cdots \ar[r]^{} & H^{p,q-1}_{\bar{\partial}}(\tilde{X}|X) \ar[d]_{(\tilde{\imath}|\imath)^{\ast}} \ar[r]^{} &H^{p,q}_{\bar{\partial}}(X) \ar[d]_{\imath^{\ast}} \ar[r]^{\pi^{\ast}} & H^{p,q}_{\bar{\partial}}(\tilde{X})
   \ar[d]_{\tilde{\imath}^{\ast}} \ar[r]^{} & H^{p,q}_{\bar{\partial}}(\tilde{X}|X)\ar[d]_{(\tilde{\imath}|\imath)^{\ast}} \ar[r]^{} & \cdots \\
   \cdots \ar[r] & H^{p,q-1}_{\bar{\partial}}(E|Z) \ar[r]^{} &
  H^{p,q}_{\bar{\partial}}(Z) \ar[r]^{\pi^{\ast}_{E}} &
  H^{p,q}_{\bar{\partial}}(E)\ar[r]^{} &
  H^{p,q}_{\bar{\partial}}(E|Z) \ar[r] & \cdots. }
\end{equation}
As the morphisms $\pi^{\ast}$ and $\pi^{\ast}_{E}$ in \eqref{comm-diagram3.1} are injective, it can split into the commutative diagram of short exact sequences
\begin{equation}\label{comm-diagram3.2}
\xymatrix@C=0.5cm{
   0 \ar[r]^{} &H^{p,q}_{\bar{\partial}}(X) \ar[d]_{\imath^{\ast}} \ar[r]^{\pi^{\ast}} & H^{p,q}_{\bar{\partial}}(\tilde{X})
   \ar[d]_{\tilde{\imath}^{\ast}} \ar[r]^{} & H^{p,q}_{\bar{\partial}}(\tilde{X}|X)
   \ar[d]_{(\tilde{\imath}|\imath)^{\ast}} \ar[r]^{} & 0 \\
   0 \ar[r]^{} &
  H^{p,q}_{\bar{\partial}}(Z) \ar[r]^{\pi^{\ast}_{E}} &
  H^{p,q}_{\bar{\partial}}(E)\ar[r]^{} &
  H^{p,q}_{\bar{\partial}}(E|Z) \ar[r] & 0. }
\end{equation}
According to the Snake Lemma \cite[Page 120]{GM03}, the diagram \eqref{comm-diagram3.2} determines an exact sequence
\begin{equation}\label{ker-coker}
\xymatrix@C=0.5cm{
  0 \ar[r] & \ker\,(\imath^{\ast})
  \ar[r]^{} & \ker\,(\tilde{\imath}^{\ast})
  \ar[r]^{} & \ker\,(\tilde{\imath}|\imath)^{\ast}
  \ar[r]^{} & \mathrm{coker}\,(\imath^{\ast})
  \ar[r]^{} & \mathrm{coker}\,(\tilde{\imath}^{\ast}) \ar[r]^{} & \mathrm{coker}\,(\tilde{\imath}|\imath)^{\ast} \ar[r] & 0 .}
\end{equation}
J. Stelzig proved that the morphism $(\tilde{\imath}|\imath)^{\ast}$ is isomorphic in \cite[Theorem 8]{Jo18}.
So the exactness in \eqref{ker-coker} implies
the isomorphisms
\begin{equation}\label{ker-coker-2}
\ker\,(\imath^{\ast})\cong\ker\,(\tilde{\imath}^{\ast})\,\,\,
\mathrm{and}\,\,\,
\mathrm{coker}\,(\imath^{\ast})
\cong\mathrm{coker}\,(\tilde{\imath}^{\ast}).
\end{equation}

From definition, the relative Dolbeault cohomology groups lie in the following commutative diagram of long exact sequences
\begin{equation}\label{comm-diagram-re-dol}
\xymatrix@C=0.5cm{
   \cdots \ar[r]^{} & H^{p,q-1}_{\bar{\partial}}(X) \ar[d]_{\pi^{\ast}} \ar[r]^{\imath^{\ast}_{q-1}}&
   H^{p,q-1}_{\bar{\partial}}(Z) \ar[d]_{\pi^{\ast}_{E}} \ar[r]^{\delta}
   & H^{p,q}_{\bar{\partial}}(X,Z) \ar[d]_{\pi^{\ast}} \ar[r]^{j_{q}} &H^{p,q}_{\bar{\partial}}(X) \ar[d]_{\pi^{\ast}} \ar[r]^{\imath^{\ast}_{q}} & H^{p,q}_{\bar{\partial}}(Z)\ar[d]_{\pi_{E}^{\ast}} \ar[r]^{} & \cdots \\
   \cdots \ar[r] & H^{p,q-1}_{\bar{\partial}}(\tilde{X}) \ar[r]^{\tilde{\imath}_{q-1}^*} &
   H^{p,q-1}_{\bar{\partial}}(E) \ar[r]^{\tilde{\delta}} &
   H^{p,q}_{\bar{\partial}}(\tilde{X},E) \ar[r]^{\tilde{j}_{q}}&
  H^{p,q}_{\bar{\partial}}(\tilde{X}) \ar[r]^{\tilde{\imath}^{\ast}_{q}} &
  H^{p,q}_{\bar{\partial}}(E)\ar[r]^{}  & \cdots. }
\end{equation}
Using the standard splitting method in homological algebra,
we can split \eqref{comm-diagram-re-dol} to be commutative diagram of short exact sequences
\begin{equation}\label{comm-diagram-re-dol2}
\xymatrix@C=0.5cm{
   0 \ar[r]^{} &\ker\,(j_{q}) \ar[d]_{\pi^{\ast}} \ar[r]^{} & H^{p,q}_{\bar{\partial}}(X,Z)\ar[d]_{\pi^{\ast}} \ar[r]^(0.6){j_{q}} & \mathrm{Im}\,(j_{q})\ar[d]_{\pi^{\ast}} \ar[r]^{} & 0\\
   0 \ar[r]^{} &
  \ker\,(\tilde{j}_{q}) \ar[r]^{} &
  H^{p,q}_{\bar{\partial}}(\tilde{X},E)\ar[r]^(0.6){\tilde{j}_{q}} &
  \mathrm{Im}\,(\tilde{j}_{q}) \ar[r] & 0. }
\end{equation}
From the exactness, we get $\mathrm{Im}\,(j_{q})=\ker\,(\imath^{\ast}_{q})$
and
$\mathrm{Im}\,(\tilde{j}_{q})=\ker\,(\tilde{\imath}^{\ast}_{q})$.
Moreover, the exactness implies the equalities:
$$\ker\,(j_{q})
  = \mathrm{Im}\,(\delta)
  \cong H^{p,q-1}_{\bar{\partial}}(Z)/\ker\,(\delta)
  = H^{p,q-1}_{\bar{\partial}}(Z)/\mathrm{Im}\,(\imath^{\ast}_{q-1})  =\mathrm{coker}\,(\imath^{\ast}_{q-1})
$$
and similarly
$\ker\,(\tilde{j}_{q})=\mathrm{coker}\,(\tilde{\imath}^{\ast}_{q-1}).$
This implies that the diagram \eqref{comm-diagram-re-dol2} is isomorphic to
$$\label{comm-diagram-re-dol3}
\xymatrix@C=0.5cm{
   0 \ar[r]^{} &\mathrm{coker}\,(\imath^{\ast}_{q-1}) \ar[d]_{\pi^{\ast}_{E}} \ar[r]^{} & H^{p,q}_{\bar{\partial}}(X,Z)\ar[d]_{\pi^{\ast}} \ar[r]^(0.6){j_{q}} & \ker\,(\imath^{\ast}_{q})\ar[d]_{\pi^{\ast}} \ar[r]^{} & 0\\
   0 \ar[r]^{} &
  \mathrm{coker}\,(\tilde{\imath}^{\ast}_{q-1}) \ar[r]^{} &
  H^{p,q}_{\bar{\partial}}(\tilde{X},E)\ar[r]^(0.6){\tilde{j}_{q}} &
  \ker\,(\tilde{\imath}^{\ast}_{q}) \ar[r] & 0. }
$$
Due to \eqref{ker-coker-2} we finally obtain that
$\pi^{\ast}:H^{p,q}_{\bar{\partial}}(X,Z)\rightarrow H^{p,q}_{\bar{\partial}}(\tilde{X},E)$ is isomorphic.
\begin{rem}
The above proof shows that $H^{p,q}_{\bar{\partial}}(\tilde{X}|X)\cong H^{p,q}_{\bar{\partial}}(E|Z)$ implies $
H^{p,q}_{\bar{\partial}}(X,Z)\cong H^{p,q}_{\bar{\partial}}(\tilde{X},E)
$. Actually, the converse still holds just by the commutative diagram
\begin{equation*}
\xymatrix{
     & &
   0 \ar[d]_{} & 0 \ar[d]_{}  &   &  \\
   \cdots \ar[r]^{} & H^{p,q}_{\bar{\partial}}(X,Z) \ar[d]_{\pi^{\ast}}^\cong \ar[r]^{}&
   H^{p,q}_{\bar{\partial}}(X) \ar[d]_{\pi^{\ast}} \ar[r]^{\imath^{\ast}}
   & H^{p,q}_{\bar{\partial}}(Z) \ar[d]_{\pi^{\ast}_{E}} \ar[r]^{\delta\quad} &H^{p,q+1}_{\bar{\partial}}(X,Z) \ar[d]_{\pi^{\ast}}^\cong \ar[r]^{} & \cdots \\
   \cdots \ar[r] & H^{p,q}_{\bar{\partial}}(\tilde{X},E) \ar[r]^{} &
   H^{p,q}_{\bar{\partial}}(\tilde{X})\ar[d] \ar[r]^{\tilde{\imath}^{\ast}} &
   H^{p,q}_{\bar{\partial}}(E)\ar[d] \ar[r]^{\tilde{\delta}\quad}&
  H^{p,q+1}_{\bar{\partial}}(\tilde{X},E) \ar[r]^{} &\cdots\\
  & &H^{p,q}_{\bar{\partial}}(\tilde{X}|X)\ar[d]\ar[r]
  &  H^{p,q}_{\bar{\partial}}(E|Z)\ar[d] &\\
  & &0 &0 & }
\end{equation*}
and Proposition \ref{cok-is}.
\end{rem}

\begin{rem}\label{meng-rk}
It is interesting to prove Proposition \ref{tech-prop} directly by the isomorphisms $(i)-(iii)$ in the proof of \cite[Theorem $8$]{Jo18}. This is completed by L. Meng in his updated version of \cite{meng1} immediately after we sent the updating for \cite{ryy}$_{v3}$ with this suggestion. Now we state this by use of our notations in \cite{ryy}$_{v3}$. Let $\mathscr{A}_{X}^{p,q}$ be the sheaf of differential $(p,q)$-forms on $X$ and similarly for $\mathscr{A}_{Z}^{p,q}$.
Set $\mathscr{E}^{p,q}_{X}=\ker\,(\varphi)$ of the surjective sheaf morphism $\varphi:\mathscr{A}^{p,q}_{X}\rightarrow \imath_{*}\mathscr{A}^{p,q}_{Z}$ in \cite[Lemma $3.9$]{ryy}$_{v3}$, and $\mathscr{F}^{p}_{X}=\ker\,(\bar{\partial}:\mathscr{E}^{p,0}_{X}\rightarrow \mathscr{E}^{p,1}_{X})$. One can define $\mathscr{F}^{p}_{\tilde X}$ similarly. By the isomorphisms $(i)-(iii)$ in the proof of \cite[Theorem $8$]{Jo18} and the long exact sequence of direct image sheaves,
Meng proved the equalities for the direct image sheaves
\begin{equation}\label{meng}
R^q\pi_*\mathscr{F}^{p}_{\tilde X}=
\begin{cases}
 \mathscr{F}^{p}_{X},\ &\text{$q=0$,} \\
 0,\ &\text{$q\geq 1$},
\end{cases}
\end{equation}
where the first equality was first given in \cite[Lemma $3.10$]{ryy}$_{v3}$. Then by \eqref{meng} and Leray spectral sequence, one completes the proof of
$$
H^{q}(X,\mathscr{F}^{p}_{X})\cong
H^{q}(X,\pi_{*}\mathscr{F}^{p}_{\tilde{X}})\cong
H^{q}(\tilde{X},\mathscr{F}^{p}_{\tilde{X}}),
$$
that is exactly the isomorphism \cite[$(3.12)$]{ryy}$_{v3}$ of cohomologies for the $\Gamma$-acyclic resolutions of the sheaves $\mathscr{F}^{p}_{X}, \mathscr{F}^{p}_{\tilde{X}}$. This immediately yields Proposition \ref{tech-prop}.
\end{rem}


\section{Applications of Main Theorem \ref{main-thm}}\label{amt}
We will present the proofs of the direct Corollaries \ref{app1}-\ref{hochschild} from Theorem \ref{main-thm} on bimeromorphic geometry of compact complex manifolds.

One starts this section with several basic notions in bimeromorphic geometry.
A nice reference of bimeromorphic geometry is \cite[$\S$ 2]{u}.
The first one is the proper modification.
\begin{defn}
A morphism $\pi: \tilde{X}\rightarrow X$ of two equidimensional complex spaces is called a \emph{proper modification}, if it satisfies:
\begin{enumerate}
  \item [$(i)$] $f$ is proper and surjective;
  \item [$(ii)$] there exist nowhere dense analytic subsets $\tilde E\subseteq \tilde{X}$ and $E \subseteq X$ such that
                  $$
                  \pi:\tilde{X}-\tilde E\rightarrow X-E
                  $$
                  is a biholomorphism, where $\tilde E:=\pi^{-1}(E)$ is called the \emph{exceptional space of the modification}.
\end{enumerate}
If $\tilde X$ and $X$ are compact, a proper modification $\pi: \tilde{X}\rightarrow X$ is often called simply a \emph{modification}.
\end{defn}

More generally, we have the following definition.
\begin{defn}\label{bimero}
Let $X$ and $Y$ be two complex spaces.
A map $\varphi$ of $X$ into the power set of $Y$ is a \emph{meromorphic map} of $X$ into $Y$,
denoted by $\varphi: X\dashrightarrow Y$, if $X$ satisfies the following conditions:
\begin{enumerate}
  \item [$(i)$] The graph $G_{\varphi}=\{(x,y)\in X\times Y\ |\ y\in \varphi(x)\}$ of $\varphi$ is an irreducible analytic subset in $X\times Y$;
  \item [$(ii)$] The projection map $P_X:G_{\varphi}\rightarrow X$ is a proper modification.
\end{enumerate}

A meromorphic map $\varphi: X\dashrightarrow Y$ of complex varieties is called a \emph{bimeromorphic map} if $P_Y:G_{\varphi}\rightarrow Y$ is also a proper modification.

If $\varphi$ is a bimeromorphic map, the analytic set
$$
\{(y,x)\in Y\times X\ |\ (x,y)\in G_\varphi\}\subseteq Y\times X
$$
defines a meromorphic map $\varphi^{-1}:Y\dashrightarrow X$ such that $\varphi\circ\varphi^{-1}=id_Y$ and $\varphi^{-1}\circ\varphi=id_X$.

Two complex varieties $X$ and $Y$ are called \emph{bimeromorphically equivalent} (or \emph{bimeromorphic}) if there exists a bimeromorphic map $\varphi: X\dashrightarrow Y$.
\end{defn}

\begin{proof}[Proof of Corollary \ref{app1}]
According to weak factorization Theorem \ref{wft}, it suffices to prove it under blow-ups.
Without loss of generality, we assume that
$$
\pi: \tilde{X} \rightarrow X
$$
is a blow-up of $X$ along a closed complex submanifold $Z\subseteq X$ of codimension $r\geq2$.
Then by Theorem \ref{main-thm}, we have
$$
H^{p,q}_{\bar{\partial}}(\tilde{X}) \cong  H^{p,q}_{\bar{\partial}}(X)\oplus \Big(\bigoplus_{k=1}^{r-1} H^{p-k,q-k}_{\bar{\partial}}(Z) \Big).
$$
In the above formula, if $p=0$ or $q=0$, then
$$
\bigoplus_{k=1}^{r-1} H^{p-k,q-k}_{\bar{\partial}}(Z)=0;
$$
otherwise, it will not be zero in general.
As a consequence, the Dolbeault cohomology isomorphisms \eqref{iso-p00q'} hold.
\end{proof}

\begin{ex}
Hodge numbers of general types are not necessarily bimeromorphic invariants.
Here is a canonical example.
Let $X$ be a projective manifold of dimension $n$.
Choose a point $x\in X$ and denote by $Bl_xX$ the blow up of $X$ at $x$.
This is a projective manifold with a holomorphic  map $\pi:Bl_xX\rightarrow X$,
which is a biholomorphism over $X-\{x\}$ such that $\pi^{-1}(x)\cong \mathbb{CP}^{n-1}$.
Then a classical calculation shows that the Hodge numbers of $Bl_xX$ are given by
\begin{align*}
h^{p,q}(Bl_xX)
 &=h^{p,q}(\mathbb{CP}^{n-1})+h^{p,q}(X-\{x\})\\
 &=h^{p,q}(\mathbb{CP}^{n-1})+h^{p,q}(X)-h^{p,q}(\{x\})\\
 &=\begin{cases}
h^{p,q}(X)+1,\ &\text{if $p=q>0$}; \\
h^{p,q}(X), &\text{otherwise}.
\end{cases}
\end{align*}
One can also use our Main Theorem \ref{main-thm} to complete this simple calculation.
\end{ex}

\begin{proof}[Proof of Corollary \ref{enbb}]
As a direct corollary of Theorem \ref{main-thm},
the blow-up $\pi:\tilde{Y}\rightarrow Y$ of a compact complex manifold $Y$ with a smooth center $Z$ satisfies that
\begin{equation*}
H^{1,1}_{\bar{\partial}}(\tilde{Y}) \cong  H^{1,1}_{\bar{\partial}}(Y)\oplus H^{0,0}_{\bar{\partial}}(Z)
\end{equation*}
and thus
$$
h^{1,1}(\tilde{Y})= h^{1,1}(Y)+1.
$$
So we obtain the equalities
$$
h^{1,1}(X_{i-1})=
\begin{cases}
h^{1,1}(X_{i})+1,\ &\text{$X_{i-1}$ is a blow-up of $X_{i}$ with a smooth center;}\\
h^{1,1}(X_{i})-1,\ &\text{$X_{i-1}$ is a blow-down of $X_{i}$ with a smooth center},
\end{cases}
$$
which imply that
$$
h^{1,1}(\tilde{X})= h^{1,1}(X)+\sharp\text{\{blow-ups in \eqref{wf}\}}-\sharp\text{\{blow-downs in \eqref{wf}\}}.
$$
\end{proof}

\begin{proof}[Proof of Theorem \ref{bsf}]
Recall that a quick definition by Poincar\'{e} and Serre dualities for \emph{degeneracy of the Fr\"olicher spectral sequence at $E_1$} on an $n$-dimensional compact complex manifold $M$ is
$$
b_k(M)=\sum_{p+q=k}h^{p,q}(M),\ \text{for each nonnegative integer $k\leq n$},
$$
where $b_k(M)$ is the $k$-th Betti number of $M$.

As a direct application of Theorem \ref{main-thm} and the formula \eqref{1}, there hold the equalities
\begin{equation}\label{buhn}
h^{p,q}(\tilde{X})= h^{p,q}(X)+\sum_{l=1}^{r-1}h^{p-l,q-l}(Z),
\end{equation}
\begin{equation}\label{bubn}
b_k(\tilde{X})= b_k(X)+\sum_{l=1}^{r-1}b_{k-2l}(Z).
\end{equation}
As usual, we assume $r\geq 2$.
Combining \eqref{buhn} with \eqref{bubn}, one has
\begin{equation}\label{bhbhz}
b_k(\tilde{X})-\sum_{p+q=k}h^{p,q}(\tilde{X})=\left(b_k(X)-\sum_{p+q=k}h^{p,q}(X)\right)+\sum_{l=1}^{r-1}\left(b_{k-2l}(Z)-\sum_{p+q=k}h^{p-l,q-l}(Z)\right).
\end{equation}
In particular, for $k=1,\cdots,4$,
\begin{equation}\label{14}
\begin{cases}
b_1(\tilde{X})-h^{1,0}(\tilde{X})-h^{0,1}(\tilde{X})=b_1(X)-h^{1,0}(X)-h^{0,1}(X),\\
b_2(\tilde{X})-h^{2,0}(\tilde{X})-h^{1,1}(\tilde{X})-h^{0,2}(\tilde{X})=b_2(X)-h^{2,0}(X)-h^{1,1}(X)-h^{0,2}(X),\\
b_3(\tilde{X})-\underset{p+q=3}{\sum}h^{p,q}(\tilde{X})=\left(b_3(X)-\underset{p+q=3}{\sum}h^{p,q}(X)\right)+\left(b_1(Z)-h^{1,0}(Z)-h^{0,1}(Z)\right),\\
b_4(\tilde{X})-\underset{p+q=4}{\sum}h^{p,q}(\tilde{X})=\left(b_4(X)-\underset{p+q=4}{\sum}h^{p,q}(X)\right)+\overset{r-1}{\underset{l=1}{\sum}} \left(b_{4-2l}(Z)-\underset{p+q=4}{\sum}h^{p-l,q-l}(Z)\right).
\end{cases}
\end{equation}

Now we assume that the Fr\"olicher spectral sequence of $\tilde{X}$ degenerates at $E_1$ and prove the first assertion.
Apply the useful Fr\"olicher inequality for an $n$-dimensional compact complex manifold
$$
b_k(X)\leq \sum_{p+q=k}h^{p,q}(X),\ k=0,\cdots,n,
$$
to obtain, for $k=1,\cdots,n,$
$$
b_k(X)-\sum_{p+q=k}h^{p,q}(X),\ b_{k-2l}(Z)-\sum_{p+q=k}h^{p-l,q-l}(Z)\leq 0,\ l=1,\cdots,r-1.
$$
By \eqref{bhbhz},
$$
0=\left(b_k(X)-\sum_{p+q=k}h^{p,q}(X)\right)+\sum_{l=1}^{r-1}\left(b_{k-2l}(Z)-\sum_{p+q=k}h^{p-l,q-l}(Z)\right)\leq 0
$$
implies
$$
b_k(X)-\sum_{p+q=k}h^{p,q}(X)=b_{k-2l}(Z)-\sum_{p+q=k}h^{p-l,q-l}(Z)=0,\ k=1,\cdots,n,\ l=1,\cdots,r-1.
$$
Thus, the Fr\"olicher spectral sequences of both ${X}$ and $Z$ degenerate at $E_1$.
The converse is similar.

Next one proceeds to the second assertion.
Using the weak factorization Theorem \ref{wft}, one reduces the argument to each blow-up.
By \eqref{bhbhz} and \eqref{14}, we just need the standard results on the degeneracy of the Fr\"olicher spectral sequences for any point, curve and surface at $E_1$ (cf. \cite[Theorem IV.2.8]{BHPV}).
\end{proof}

\begin{rem}\label{b2h11}
From the first two equalities of \eqref{14}, it follows that the quantities $b_2(M)-h^{1,1}(M)$ and $b_1(M)$ are bimeromorphic invariants of a compact complex manifold $M$.
Nevertheless, analogously to the proof of Corollary \ref{enbb}, one obtains
$$
b_2(\tilde{X})-b_2(X)=\sharp\text{\{blow-ups in \eqref{wf}\}}-\sharp\text{\{blow-downs in \eqref{wf}\}}=h^{1,1}(\tilde{X})-h^{1,1}(X).
$$
\end{rem}
\begin{rem}
It is easy to see from \eqref{bhbhz} that the second assertion holds for any dimension if all compact complex submanifolds of codimensions at least two in a compact complex manifold with the degeneracy of the Fr\"olicher spectral sequences at $E_1$ still admit this degeneracy.
Compare also the argument in \cite[Theorem 2.1, Question 2.4 and Remark 2.6]{astt}
for Question \ref{quest}.
\end{rem}

\begin{rem}
It is interesting to construct a compact complex manifold such that its Fr\"olicher spectral sequence degenerates at $E_1$ and the Hodge symmetry $H^{p,q}(-)\cong H^{q,p}(-)$ {for all possible $p, q$} holds on it,
but it does not satisfy the $\partial\bar{\partial}$-lemma, as provided recently in \cite[Proposition $4.3$]{couv}.
From our blow-up formula for Dolbeault cohomologies, we notice that the Hodge symmetry is a bimeromorphic property for compact complex threefolds,
while fortunately, the $\partial\bar{\partial}$-lemma is also a bimeromorphic property for compact complex threefolds.

In this way, by bimeromorphic transformations, we can construct many more examples of compact non-$\partial\bar{\partial}$-threefolds with the degeneracy of Fr\"olicher spectral sequences at $E_1$ and Hodge symmetry from the known ones, such as the one in \cite[Proposition $4.3$]{couv}.
\end{rem}

\begin{proof}[Proof of Corollary \ref{hochschild}]
This is a direct consequence of Theorem \ref{main-thm} and Hochschild-Kostant-Rosenberg (HKR) theorem for complex manifolds (cf. \cite[Corollary 4.2)]{Cal05}).
In fact,
we have the following isomorphisms:
\begin{eqnarray*}
\mathrm{HH}_{k}(\tilde{X})
&\cong& \bigoplus_{p-q=k} H^{q}(\tilde{X}, \Omega_{\tilde{X}}^p) \;\;(\textrm{HKR theorem for $\tilde{X}$})\\
&\cong &\bigoplus_{p-q=k}  \Big( H^{p,q}_{\bar{\partial}}(X)\oplus \bigoplus_{i=1}^{r-1} H^{p-i,q-i}_{\bar{\partial}}(Z) \Big) \;\; (\textrm{by Theorem \ref{main-thm}}) \\
&\cong & \bigoplus_{p-q=k} H^{q}(X, \Omega_{X}^p) \oplus
\bigoplus_{i=1}^{r-1} \Big( \bigoplus_{p-q=k} H^{q-i}(Z, \Omega_{Z}^{p-i}) \Big)\\
&\cong & \bigoplus_{p-q=k} H^{q}(X, \Omega_{X}^p) \oplus
\Big( H^{0}(Z, \Omega_{Z}^{k})\oplus\cdots\oplus H^{n-k}(Z, \Omega_{Z}^{n})\Big)^{\oplus (r-1)}\\
&\cong & \mathrm{HH}_{k}(X) \oplus  \Big( \mathrm{HH}_{k}(Z) \Big)^{\oplus (r-1)}\ \;\;(\textrm{by HKR theorem for $X$ and $Z$}),
\end{eqnarray*}
for any $-n\leq k\leq n$.
\end{proof}

\begin{rem}
As we know the Hochschild homologies are important invariants of compact complex manifolds.
For instance, the Hochschild homology is a derived invariant,
that is, for two compact complex manifolds (or in particular smooth projective varieties) $X$ and $Y$,
if the derived category $\mathcal{D}_{coh}^b(X)$
\footnote{For a compact complex manifold $X$, $\mathcal{D}_{coh}^b(X)$ denotes the bounded derived category of $\mathcal{O}_X$-modules with coherent cohomology.}
is equivalent to $\mathcal{D}_{coh}^b(Y)$ as triangulated categories, then the Hochschild homologies of $X$ and $Y$ are isomorphic;
see for example \cite{Cal05} and references therein.
In \cite[Theorem 4.3]{Orlov93}, Orlov obtained the blow-up formula for derived categories of smooth projective varieties.
Furthermore, if one is able to obtain that for compact complex manifolds,
then one can also get Corollary \ref{hochschild}, which is believed to be known for experts.
\end{rem}

\section{Examples of Corollary \ref{enbb}}\label{ex}
In this section, we list several examples with equal $(1,1)$-Hodge number or equivalently second Betti number for Corollary \ref{enbb}.
They are believed to be of independent interest for further study since they much concern about the relationship between Hodge structure and bimeromorphic geometry.
\begin{thm}\label{exthm}
Let $\pi:\tilde{X}\dashrightarrow X$ be a bimeromorphic map between two compact complex manifolds.
Then the numbers of blow-ups and blow-downs in the weak factorization \eqref{wf} are equal if the complex manifolds belong to one of the following:
\begin{enumerate}
    \item[$(i)$]\label{i}
           Both $\tilde{X}$ and $X$ are surfaces with nef canonical bundles;
    \item [$(ii)$]\label{ii}
           $\tilde{X}$ and $X$ are two bimeromorphic K\"ahler minimal models of dimension three with nonnegative Kodaira dimension except two;
    \item [$(iii)$]\label{iii}
           $\tilde{X}$ and $X$ are two bimeromorphic minimal models, since they are \emph{isomorphic in codimension one}, i.e., there exist closed subsets $\tilde B\subseteq \tilde{X}$ and $B\subseteq X$ of codimension at least two such that $\pi$ induces an isomorphism
           $$\tilde{X}-\tilde B\stackrel{\pi}{\simeq}X-B.$$
\end{enumerate}
\end{thm}

We first recall several notions in minimal model program.
Let $M$ be a normal variety.
We say that a normal variety $M$ is \emph{$\mathbb{Q}$-factorial} if for every Weil divisor $D$ there exists an integer $m\in \mathbb{N}$ such that $\mathcal{O}_{M}(mD)$ is a locally free sheaf, i.e.,
$mD$ is a Cartier divisor and in addition that there is some number $m\in \mathbb{N}$
such that the coherent sheaf $(K_M^{\otimes m})^{**}=(\omega_M^{\otimes m})^{**}$
on the canonical sheaf $K_M=\omega_M$ is locally free.
Then we write
$$
mK_M=(K_M^{\otimes m})^{**}.
$$
A normal variety $M$ has \emph{terminal singularities} if
\begin{enumerate}
    \item[$(i)$]\label{i}
              there is a positive integer $k$ such that $kK_M$ is a Cartier divisor;
    \item [$(ii)$]\label{ii}
              for some desingularization $f: \tilde{M}\rightarrow M$,
              any $k$-canonical form on $M_{reg}$ extends a $k$-canonical form on $\tilde{M}$ vanishing along every exceptional divisor of $\tilde{M}$, or equivalently,
              $F-E(f)$ is effective if we write
              $$kK_{\tilde M}\equiv f^*(kK_M)+F$$
              and $E(f)$ denotes the union of all reduced $f$-exceptional hypersurfaces in $\tilde M$.
 \end{enumerate}
Notice that the property of terminal singularities does not depend on the choice of desingularization,
and a smooth variety has terminal singularities.
Now let us recall the definition of nefness \cite[Definition $3$]{paun}.
Let $[\alpha]\in H^{1,1}_{BC}(M)$ be a class represented by a form $\alpha$ with local potentials.
Then $[\alpha]$ is called \emph{nef} if for some positive $(1,1)$-form $\omega$ on $M$ and every $\epsilon> 0$,
there is some smooth function $\beta_{\epsilon}$ on $M$ such that
$$
\alpha+\sqrt{-1}\partial\bar\partial \beta_{\epsilon}\geq -\epsilon\omega.
$$
A divisor $D$ on a Moishezon variety $M$ is called \emph{algebraically nef} if $D\cdot C \geq 0$ for all curves $C$ in $M$.
These two definitions coincide in this case by \cite[Corollaire on P. 412]{paun}.
A \emph{Moishezon variety} is a compact complex variety with the algebraic dimension equal to its dimension,
or equivalently it is bimeromorphically equivalent to a projective variety.

\begin{defn}\label{mm}
A \emph{minimal model} is a normal $\mathbb{Q}$-factorial variety $M$ with nef canonical divisor $K_M$ and at most terminal singularities.
\end{defn}

Let us return to Theorem \ref{exthm}.
The first item follows from a classical result in compact complex surface theory that all bimeromorphic surfaces with nef canonical bundles are isomorphic (cf. \cite[Claim on P. 99]{BHPV}).

By the remarkable work \cite{be}, Calabi-Yau manifolds, hyperk\"ahler manifolds and complex tori
form the building blocks of compact K\"ahler manifolds with vanishing first Chern classes.
Recall that a compact complex manifold $M$ is called \emph{weakly Calabi-Yau} if its canonical bundle $K_M\cong \mathcal{O}_M$.
For the full \emph{Calabi-Yau} condition one usually also requires that $M$ be simply connected and $h^q(M, \mathcal{O}_M) = 0$ for $0 < q < \dim M$, and then it is projective by Kodaira's criterion.
A complex manifold $M$ is called \emph{symplectic} (here) if there exists a holomorphic two-form $\omega$ which is non-degenerate at every point. Note that the existence of $\omega$ implies that the canonical bundle is trivial.
If $M$ is compact, then the symplectic structure is unique if and only if $h^0(M,\Omega_M^2)=1$.
By definition, a simply connected compact K\"ahler manifold with a unique symplectic structure is \emph{irreducible symplectic}.

As shown in \cite[\S 2.2]{h}, two bimeromorphic compact symplectic manifolds with unique symplectic structures are isomorphic in codimension one.
In particular, a famous theorem \cite[Corollary $4.7$]{hy} of Huybrechts implies that two birational projective irreducible
symplectic manifolds  have the same Betti numbers and Hodge numbers.
Clearly, the same holds true for complex tori.

Moreover, one has the following remarkable theorem (also cf. \cite[Corollary $4.12$]{ko} for threefolds).
\begin{thm}\label{kon}
\begin{enumerate}[$(i)$]
    \item  {{$($\emph{\cite{k,ba,w0}}$)$}\label{stab-b}
           Any two birational weakly  Calabi-Yau or more generally projective manifolds with canonical bundles nef along the exceptional loci have the same Betti numbers};
    \item  {{$($\emph{\cite{k,i,w}}}$)$}\label{stab-h}
           Any two birational smooth projective minimal models have  the same  Hodge numbers.
\end{enumerate}
\end{thm}
Recall that the \emph{exceptional locus} of a bimeromorphic map $f: Y\dashrightarrow Z$ is the points of $Y$ where $f$ is not a local isomorphism.  It is an open problem whether the analogue of \eqref{stab-b} is true for compact K\"ahler manifolds.

\subsection{Bimeromorphic minimal models for K\"ahler threefolds}
In this subsection, we will study the second item of Theorem \ref{exthm}.

\begin{prop}[{\cite[Corollary 7.3]{lin1}}]\label{lin3}
Let ${X}$ and $\tilde X$ be two bimeromorphic smooth K\"ahler minimal models of dimension three with nonnegative Kodaira dimension except two.
Then they have the same Hodge numbers.
\end{prop}
\begin{proof}
The general type case (i.e., the Kodaira dimension is the dimension of manifold) for algebraic approximation problem becomes trivial since every K\"ahler Moishezon manifold is projective and then one takes  the trivial deformation of this manifold.

Let $\pi:\tilde{\mathfrak X}\rightarrow \Delta$ be a small deformation of $X$ to some projective variety $Y$ by \cite{g,lin1,lin2}.
By assumption and \cite[Theorem 4.9]{ko}, the bimeromorphic map $X\dashrightarrow \tilde X$ is a composition of a finite sequence of flops.
Roughly speaking, a flop is a codimension-$2$ surgery operation, a sequence of which connects two minimal models in a bimeromorphic equivalence class.
It is given by removing a curve on which the canonical divisor admits degree $0$ and replacing it with another curve with the same property while there is a Cartier divisor that is negative on the first curve and positive on the second one.
By \cite[Theorem 12.6.2, Remrak 12.6.3]{lm}, up to shrinking $\Delta$, there exists a deformation
$\tilde\pi:\tilde{\mathfrak{X}}\rightarrow \Delta$ of $\tilde X$ and a bimeromorphic map $\phi: \mathfrak{X}\dashrightarrow \tilde{\mathfrak{X}}$ over $\Delta$ extending $X\dashrightarrow \tilde X$.
Let $\tilde Y$ be the image of $Y$ under $\phi$.
Then $\tilde Y$ and $Y$ are also connected by a finite sequence of flops as in the proof of \cite[Theorem 11.10]{lm}.
In summary, one obtains the diagram
$$
\xymatrix@R=0.25cm@C=0.6cm{
Y\ar[rr] \ar@{^{(}->}[ddr] && \tilde Y\ar@{^{(}->}[ddr]|\hole \\
&X\ar@{.>}[rr] \ar@{^{(}->}[d] & &\tilde X\ar@{^{(}->}[d]& \\
  &\mathfrak{X} \ar@{.>}[rr]^{\phi} \ar[dr]_{\pi}
                &  &    \tilde{\mathfrak{X}} \ar[dl]^{\tilde\pi}&    \\
               & & \Delta}.
$$

As $Y$ (resp., $\tilde Y$) is a small deformation of $X$ (resp., $\tilde X$), $\tilde Y$ is K\"ahler, Moishezon and thus projective.
Since $\tilde Y$ is a small deformation of $\tilde X$, $\tilde Y$ is K\"ahler by the fundamental Kodaira-Spencer's local stability theorem of K\"ahler structures (cf. \cite[Theorem 15]{KS} and also \cite{RwZ} for a new proof), and thus projective since it is also Moishezon.
So one obtains the equalities
$$
h^{p,q}(X)=h^{p,q}(Y)=h^{p,q}(\tilde Y)=h^{p,q}(\tilde X),
$$
where the first and third equalities follow from \cite[Proposition 9.20]{V} and also \cite[Theorem 1.3]{RZ15} for a general argument, and the second one is got by Theorem \ref{kon}.\eqref{stab-h} or Theorem \ref{kon}.\eqref{stab-b} with Corollary \ref{app1} and the bimeromorphic map between $\tilde Y$ and $Y$.
Hence, one completes the proof.
\end{proof}

\subsection{Isomorphic complex manifolds in codimension one}
This subsection is to discuss the third item of Theorem \ref{exthm}.
We first need a useful proposition.
\begin{prop}[]\label{pln}
Let $f: X\dashrightarrow Y$ be a bimeromorphic map between compact complex manifolds, which are {isomorphic in codimension one}.
Then there are natural isomorphisms
$$
H^k(X;\mathbb{Z})\rightarrow H^k(Y;\mathbb{Z}),\ \text{for $k\leq 2,$}
$$
and also
$$
\sum_{p+q=2}h^{p,q}(X)=\sum_{p+q=2}h^{p,q}(Y).
$$
\end{prop}
\begin{proof}
Here is a proof extracted from Popa's lecture notes {\cite[Proposition $1.11$ in Chapter $4$]{popa}}.
Poincar\'{e} duality implies that equivalently one can aim for natural isomorphisms
$$H_{2n-k}(X;\mathbb{Z})\rightarrow H_{2n-k}(Y;\mathbb{Z}),\ \text{for $k\leq 2$}$$
where $n$ is the complex dimension of $X$.
Now $X$ and $Y$ are diffeomorphic as real manifolds
outside closed subsets of real codimension at least four, and therefore this diffeomorphism
sees all $(2n-k)$-cycles on $X$ and $Y$ with $k\leq 2$, inducing the desired natural isomorphism.

The next assertion follows from Corollary \ref{app1} and Remark \ref{b2h11}.
\end{proof}

So as for the third item of Theorem \ref{exthm}, we just need:
\begin{thm}[{\cite[Lemma 4.3]{ko}}]\label{iic1}
Any two bimeromorphic minimal models $\tilde X$ and ${X}$ in Definition \ref{mm} under the bimeromorphic map $\pi$ are {isomorphic in codimension one}.
\end{thm}
The projective analogue of this theorem is well-known to bi-rationalists (cf. \cite[\S 7.18]{deb} for example or more recent \cite{Ka} by flops) and we outline a proof here for analytic geometer's convenience.
\begin{proof}[Proof of Theorem \ref{iic1}]
The proof heavily relies on the \lq negativity lemma' \cite[Lemma $3.39$]{lm2}:
Let $h:{Z}\rightarrow Y$ be a projective bimeromorphic  morphism between normal varieties and $-D$ an $h$-nef $\mathbb Q$-Cartier $\mathbb{Q}$-divisor on $Z$.
Then $h_*D$ is effective if and only if $D$ is.
Recall that a divisor on a normal variety is \emph{$f$-nef} for a projective morphism $f$ if it has nonnegative intersection with every curve contracted by $f$.
One applies a hyperplane section argument to reduce its proof to the surface case originally by \cite{gr,mum}.

By assumption, there exists a smooth complex variety $W$ with two projective bimeromorphic  morphisms $\jmath: W \dashrightarrow \tilde X$ and $\imath: W \dashrightarrow X$, and effective $\mathbb Q$-divisors $\tilde F$ and $F$ such that
$$
K_W\sim \jmath^*K_{\tilde X}+\tilde F\sim \imath^*K_X+F.
$$
Set $D=F-\tilde F$.
For any curve $C$ contracted by $\jmath$, one has
$$
D\cdot C=(\jmath^* K_{\tilde X}-\imath^* K_{X})\cdot C=-K_{X}\cdot \imath_* C\leq 0
$$
since $K_{X}$ is nef.
As $\jmath_*D=\jmath_* F$ is effective, the negativity lemma implies that $D$ is effective and thus $F\geq \tilde F$.
Since $X$ has terminal singularities, any $\jmath$-exceptional divisor appear in $\tilde F$ and also in $F$.
It is thus $\imath$-exceptional and implies that $\imath(Exc(\jmath))$ has codimension at least two.
Here $Exc(\jmath)$ denotes the exceptional locus of $\jmath$.
Analogously, $\jmath(Exc(\imath))$ has codimension at least two.
Hence,
$
\tilde X-\jmath(Exc(\jmath)\cup Exc(\imath))
$
and
$
X-\imath(Exc(\jmath)\cup Exc(\imath))
$
are isomorphic.
\end{proof}

\begin{rem}
The \lq nefness' assumption for the canonical divisors in Theorem \ref{iic1} can be weakened as \lq the canonical divisors are nef along the exceptional loci'.
\end{rem}


\appendix

\section{Blow-up formula for de Rham cohomologies}\label{appendix}

In this appendix, we give a new proof of the blow-up formula \eqref{1} for de Rham cohomologies by use of the relative de Rham cohomology in the sense of Godbillon \cite[Chapitre XII]{Go97}.
One finds that the de Rham case is much easier than the Dolbeault one. The easier thing here is the existence of smooth tubular neighborhood on the smooth manifolds while  holomorphic tubular neighborhood does not necessarily exist (even on the K\"ahler manifolds, cf. \cite{rw} and the references therein).

Assume that $M$ is a smooth manifold with dimension $n$ and let $N$ be a $k$-dimensional closed submanifold of $M$.
Consider the space of differential forms
$$
\mathcal{A}^{\bullet}(M,N)=\{\alpha\in\mathcal{A}^{\bullet}(M)\,|\,i^{\ast}\alpha=0\},
$$
where $i^{\ast}$ is the pullback of the inclusion $i:N\hookrightarrow M$.
Since $\mathcal{A}^{\bullet}(M,N)$ is closed under the action of the exterior differential operator $d$ we get a sub-complex of the de Rham complex $\{\mathcal{A}^{\bullet}(M),d\}$ which is called the \emph{relative de Rham complex} with respect to $N$:
$$
\xymatrix{
0 \ar[r]^{} & \mathcal{A}^{0}(M,N) \ar[r]^{d} & \mathcal{A}^{1}(M,N)  \ar[r]^{d} & \mathcal{A}^{2}(M,N) \ar[r]^{\quad d} & \cdots.}
$$
The associated cohomology, denoted by $H_{dR}^{\bullet}(M,N)$, is called the \emph{relative de Rham cohomology} of the pair $(M,N)$.
From definition, it is straightforward to verify that if $p>k$ then $H^{p}_{dR}(M,N)=H^{p}_{dR}(M)$.
In particular, there exists a short exact sequence of complexes
$$\label{short-exact-relative}
\xymatrix@C=0.5cm{
0 \ar[r]^{} & \mathcal{A}^{\bullet}(M,N) \ar[r]^{} & \mathcal{A}^{\bullet}(M)  \ar[r]^{i^{\ast}} & \mathcal{A}^{\bullet}(N) \ar[r]^{} & 0},
$$
which yields a long exact sequence
$$
\xymatrix@C=0.5cm{
\cdots\ar[r]^{} & H^{\bullet}_{dR}(M,N) \ar[r]^{} & H^{\bullet}_{dR}(M)  \ar[r]^{} & H^{\bullet}_{dR}(N) \ar[r]^{} & H^{\bullet+1}_{dR}(M,N) \ar[r]^{} & \cdots}.
$$

From now on, we follow the notations in Subsection \ref{sec-bu}.
Set $U:=X-Z$ and let $\jmath: U\rightarrow X$ be the inclusion.
Let $\mathcal{A}^{\bullet}(X,Z)$ be the \emph{relative de Rham complex}.
Then we obtain a short exact sequence
$$\label{mainexactseq1}
\xymatrix@C=0.5cm{
0 \ar[r]^{} & \mathcal{A}^{\bullet}(X,Z) \ar[r]^{} & \mathcal{A}^{\bullet}(X)  \ar[r]^{\imath^{\ast}} & \mathcal{A}^{\bullet}(Z) \ar[r]^{} & 0}
$$
and analogously,
$$\label{mainexactseq2}
\xymatrix@C=0.5cm{
0 \ar[r]^{} & \mathcal{A}^{\bullet}(\tilde{X},E) \ar[r]^{} & \mathcal{A}^{\bullet}(\tilde{X})  \ar[r]^{\tilde{\imath}^{\ast}} & \mathcal{A}^{\bullet}(E) \ar[r]^{} & 0.}
$$
In particular, the blow-up diagram \eqref{blow-up-diag3} induces a commutative diagram of short exact sequences
\begin{equation}\label{s-e-diagram}
\xymatrix@C=0.5cm{
 0 \ar[r]^{} & \mathcal{A}^{\bullet}(X,Z) \ar[d]_{\pi^{* }} \ar[r]^{} & \mathcal{A}^{\bullet}(X) \ar[d]_{\pi^{* }} \ar[r]^{\imath^{* }} &  \mathcal{A}^{\bullet}(Z)\ar[d]_{\pi_{E}^*} \ar[r]^{} & 0 \\
0 \ar[r] &  \mathcal{A}^{\bullet}(\tilde{X},E) \ar[r]^{} &
   \mathcal{A}^{\bullet}(\tilde{X}) \ar[r]^{\tilde{\imath}^{* }} &
   \mathcal{A}^{\bullet}(E) \ar[r] &0. }
\end{equation}
Then the commutative diagram \eqref{s-e-diagram} gives a commutative diagram of long exact sequences
\begin{equation}\label{comm-diagram1}
\xymatrix@C=0.5cm{
   \cdots \ar[r]^{} & H^{k}_{dR}(X,Z) \ar[d]_{\pi^{* }} \ar[r]^{} &H^{k}_{dR}(X) \ar[d]_{\pi^{* }} \ar[r]^{} & H^{k}_{dR}(Z)\ar[d]_{\pi_{E}^*} \ar[r]^{} & H^{k+1}_{dR}(X,Z)\ar[d]_{\pi^{* }} \ar[r]^{} & \cdots \\
   \cdots \ar[r] & H^{k}_{dR}(\tilde{X},E) \ar[r]^{} &
  H^{k}_{dR}(\tilde{X}) \ar[r]^{} &
  H^{k}_{dR}(E)\ar[r]^{} &
  H^{k+1}_{dR}(\tilde{X},E) \ar[r] & \cdots. }
\end{equation}

Let $\mathcal{A}^{\bullet}_{c}(U)$ be the compactly supported de Rham complex of $U=X-Z$.
Then the chain map
$
\jmath_{\ast}:\mathcal{A}^{\bullet}_{c}(U)\rightarrow \mathcal{A}^{\bullet}(X)
$
has the image in the relative de Rham complex $\mathcal{A}^{\bullet}(X,Z)$.
Moreover, the morphism
$
\jmath_{\ast}:\mathcal{A}^{\bullet}_{c}(U)\rightarrow \mathcal{A}^{\bullet}(X,Z)
$
is quasi-isomorphic (cf. \cite[Proposition 13.11]{MT97}), i.e., there holds the isomorphism
\begin{equation}\label{iso1}
H^{\bullet}_{dR,c}(U)\cong H^{\bullet}_{dR}(X,Z),
\end{equation}
and similarly,
\begin{equation}\label{iso2}
H^{\bullet}_{dR,c}(\tilde{U})\cong H^{\bullet}_{dR}(\tilde{X},E).
\end{equation}
As $\pi|_{\tilde{U}}:\tilde{U}\rightarrow U$ is biholomorphic, we get the isomorphism for any $l\geq 0$
\begin{equation}\label{iso3}
H^{l}_{dR,c}(U)\cong H^{l}_{dR,c}(\tilde{U}).
\end{equation}
Due to the isomorphisms \eqref{iso1}-\eqref{iso3}, the diagram \eqref{comm-diagram1} equals to
\begin{equation}\label{comm-diagram2}
\xymatrix@C=0.5cm{
   \cdots \ar[r]^{} & H^{k}_{dR,c}(U) \ar[d]^{\cong}_{\pi^{* }} \ar[r]^{} &H^{k}_{dR}(X) \ar[d]_{\pi^{* }} \ar[r]^{} & H^{k}_{dR}(Z)\ar[d]_{\pi_{E}^*} \ar[r]^{} & H^{k+1}_{dR,c}(U)\ar[d]^{\cong}_{\pi^{* }} \ar[r]^{} & \cdots \\
   \cdots \ar[r] & H^{k}_{dR,c}(\tilde{U}) \ar[r]^{} &
  H^{k}_{dR}(\tilde{X}) \ar[r]^{} &
  H^{k}_{dR}(E)\ar[r]^{} &
  H^{k+1}_{dR,c}(\tilde{U}) \ar[r] & \cdots. }
\end{equation}
Note that since
$
\pi:\tilde{X}\rightarrow X$ and $\pi_{E}:E\rightarrow Z
$
are proper surjective holomorphic maps,
$
\pi^{\ast}:H^{k}_{dR}(X)\rightarrow H^{k}_{dR}(\tilde{X})
$
and thus
$
\pi_{E}^{\ast}:H^{k}_{dR}(Z)\rightarrow H^{k}_{dR}(E)
$
are injective by \cite[Theorem 3.1]{Wells74} and the Weak Five Lemma \cite[Lemma $3.3.(i)$]{mc}, respectively.

From \eqref{comm-diagram2} and Proposition \ref{cok-is}, an isomorphism induced by the pullback of the inclusion $\tilde{\imath}:E\hookrightarrow \tilde{X}$ follows
$$\label{blow-up-0}
{H^{k}_{dR}(\tilde{X})}/{\pi^{\ast}H^{k}_{dR}(X)}\cong
{H^{k}_{dR}(E)}/{\pi_{E}^{\ast}H^{k}_{dR}(Z)}.
$$
According to the Leray-Hirsch lemma, the de Rham cohomology of $E$ is a free $H^{\bullet}_{dR}(Z)$-module with the basis $\{1,\bm {t},\cdots,{\bm t}^{r-1}\}$, where $\bm t=c_{1}(\mathcal{O}_{E}(1))$.
This yields the isomorphism
$$
{H^{k}_{dR}(E)}/{\pi_{E}^{\ast}H^{k}_{dR}(Z)}\cong
\bigoplus_{i=1}^{r-1}H^{k-2i}_{dR}(Z).
$$
Observing that $\pi^{\ast}:H^{k}_{dR}(X)\rightarrow H^{k}_{dR}(\tilde{X})$ is injective,
one obtains the de Rham blow-up formula
$$
H^{k}_{dR}(\tilde{X})
\cong H^{k}_{dR}(X)\oplus \left(\bigoplus_{i=1}^{r-1}H^{k-2i}_{dR}(Z) \right).
$$


\section*{Acknowledgement}
This work started from a conversation of Professor C. Voisin with the first author during the conference in honor of J.-P. Demailly on occasion of his $60$th birthday at Institut Fourier, Universit\'{e} Grenoble Alpes, and most of this work was completed during the first author's visit to Institute of Mathematics, Academia Sinica since September 2017.
He would like to express his gratitude to both institutes for their hospitality during his visit, especially Professors J.-P. Demailly, Jih-Hsin Cheng, Chin-Yu Hsiao.
The second and third authors would like to thank Departments of Mathematics at Universit\`{a} degli Studi di Milano and Cornell University for the hospitalities during their respective visits.
All the authors would like to express their gratitude to Professor J.-A. Chen for pointing out the reference \cite{Ka}, Dr. Jie Liu for many useful discussions on birational geometry, especially Theorem \ref{iic1}, and Professors D. Angella, Chun-Chung Hsieh, I-Hsun Tsai for their interest on our paper. Last, but not least, we sincerely thank  Dr. J. Stelzig for sending his preprint \cite{Jo18} which is quite useful for us to complete the proof of Proposition \ref{tech-prop}.

\end{document}